\documentclass[reqno,openbib,runningheads,a4paper,12pt]{amsart}%
\usepackage{graphicx}
\usepackage{amssymb}
\usepackage{amsfonts}
\usepackage{amsmath}
\usepackage{graphicx}
\usepackage{lineno}
\usepackage[authoryear]{natbib}
\usepackage[dvipsnames]{xcolor}
\usepackage{epsf}
\usepackage{lscape}
\usepackage[legalpaper,bookmarks=true,colorlinks=true,linkcolor=blue,citecolor=blue]%
{hyperref}%
\providecommand{\U}[1]{\protect\rule{.1in}{.1in}}
\providecommand{\U}[1]{\protect \rule{.1in}{.1in}}
\newtheorem{theorem}{Theorem}[section]

\newtheorem{proposition}{Proposition}[section]

\newtheorem{remark}{Remark}[section]

\makeatletter
\renewcommand{\@biblabel}[1]{}
\makeatother
\begin{document}

\begin{center}
{\large Weighted Estimation of the Tail Index under Right Censorship: A
Unified Approach Based on Kaplan-Meier and Nelson-Aalen Integrals}

{\Large \textbf{ }}\bigskip

{\large Abdelhakim Necir}$^{\ast}${\large , Nour Elhouda Guesmia, Djamel
Meraghni}\medskip\newline

{\small \textit{Laboratory of Applied Mathematics, Mohamed Khider University,
Biskra, Algeria}}\medskip\medskip
\end{center}

\noindent Kaplan-Meier and Nelson-Aalen integral estimators to the tail index
of right-censored Pareto-type data traditionally rely on the assumption that
the proportion $p$ of upper uncensored observations exceeds $1/2,$
corresponding to weak censoring regime. However, this condition excludes many
practical settings characterized by strong censorship, where $p\leq1/2.$ To
address this bothering limitation, we propose a modification that incorporates
a tuning parameter. This parameter, greater than one, assigns appropriate
weights to the estimators, thereby extending the applicability of the method
to the entire censoring range, i.e. $p\in\left(  0,1\right)  .$ Under suitable
regularity conditions, we establish the consistency and asymptotic normality
of the proposed estimators. Extensive simulation studies reveal a clear
improvement over existing methods in terms of bias and mean squared error,
particularly in the strong censoring situation. These results highlight the
significant practical and theoretical impact of our approach, offering a more
flexible and accurate framework for tail index estimation under censoring. The
usefulness of the method is further illustrated through its application to two
real datasets: one on insurance losses (weak censoring) and the other on AIDS
cases (strong censoring).\bigskip

\medskip

\noindent\textbf{AMC 2020 Subject Classification:} 62G32; 62G05; 62G20.

\vfill

\vfill

\noindent{\small $^{\text{*}}$Corresponding author:
ah.necir@univ-biskra.dz\newline\noindent\textit{E-mail address:}\newline
nourelhouda.guesmia@univ-biskra.dz\texttt{\ }(N.~Guesmia)\newline
djamel.meraghni@univ-biskra.dz (D.~Meraghni)}\newline

\section{\textbf{Introduction\label{sec1}}}

\noindent Right-censored data, characterized by partially observed values,
present a common challenge in statistics. Indeed, this type of data is only
known within a certain range and remains unspecified beyond it. Such data
arise in fields such as actuarial science, survival analysis, and reliability
engineering. Pareto-type distributions are commonly used to model data with
extreme values. Compared to the normal distribution, Pareto-type distributions
are characterized by a significantly higher probability of observing values
far from the mean. These distributions are crucial in applications such as
insurance, finance, and environmental studies, where rare events can have a
substantial impact. Let $X_{1},X_{2},...,X_{n}$ be a sample of size $n\geq1$
from a random variables (rv) $X,$ and let $C_{1},C_{2},...,C_{n}$ be another
sample from a rv $C,$ defined on a probability space $\left(  \Omega
,\mathcal{A},\mathbf{P}\right)  ,$ with continuous cumulative distribution
functions (cdf's) $F$ and $G,$ respectively. Assume that rv's $X$\ and $C$ are
independent. We consider the classical right-censoring framework, where the
variable of interest $X$ is subject to censoring by $C.$ Specifically, for
each observation stage $1\leq j\leq n,$ we observe the pair $\left(
Z_{j},\delta_{j}\right)  $ where $Z_{j}:=\min\left(  X_{j},C_{j}\right)  $ and
$\delta_{j}:=\mathbb{I}_{\left\{  X_{j}\leq C_{j}\right\}  },$ with
$\mathbb{I}_{\left\{  A\right\}  }$ denoting the indicator function of the set
$A.$ The rv $\delta_{j}$ thus indicates whether the $j$-th observation is
uncensored. Throughout the paper, we will use the notation $\overline
{\mathcal{S}}(x):=\mathcal{S}(\infty)-\mathcal{S}(x),$ for any measurable
function $\mathcal{S}.$ We assume that the tail functions $\overline{F}$ and
$\overline{G}$ are regularly varying at infinity with positive tail indices
$\gamma_{1}>0$ and $\gamma_{2}>0$ respectively (or that cdf's $F$ and $G$ are
Pareto-like). In other words, for any $x>0,$ we have%
\begin{equation}
\lim_{t\rightarrow\infty}\frac{\overline{F}(tx)}{\overline{F}(t)}%
=x^{-1/\gamma_{1}}\text{ and }\lim_{t\rightarrow\infty}\frac{\overline{G}%
(tx)}{\overline{G}(t)}=x^{-1/\gamma_{2}}. \label{RV}%
\end{equation}
If we denote the cdf of $Z$ by $H,$ then by the independence between $X$ and
$C,$ we have $\overline{H}=\overline{F}\times\overline{G},$ which yields that
$\overline{H}$ is also regularly varying at infinity, with tail index
$\gamma:=\gamma_{1}\gamma_{2}/\left(  \gamma_{1}+\gamma_{2}\right)  .$ In the
presence of right-censored data with extreme values, several methods have been
proposed for estimating the tail index. Numerous techniques have been
developed to address the specific challenges posed by the nature of such data.
Many studies have focused on modifying traditional tail index estimation
procedures, such as Hill's method. \citep[][]{Hill75}, to make them suitable
for censored data. For instance, \cite{EnFG08} adapted the latter to handle
the characteristics of right-censored data. Their estimator is given by%
\[
\widehat{\gamma}_{1,k}^{\left(  EFG\right)  }:=\frac{\widehat{\gamma}%
_{k}^{\left(  H\right)  }}{\widehat{p}_{k}},
\]
where $\widehat{\gamma}_{k}^{\left(  H\right)  }:=k^{-1}%
{\displaystyle\sum\limits_{i=1}^{k}}
\log\left(  Z_{n-i+1:n}/Z_{n-k:n}\right)  ,$ denotes the popular Hill
estimator of the tail index $\gamma$ corresponding to the complete $Z$-sample
and $\widehat{p}_{k}:=k^{-1}%
{\displaystyle\sum_{i=1}^{k}}
\delta_{\left[  n-i+1:n\right]  },$ stands for an estimator of the proportion
of upper uncensored observations $p:=\gamma_{2}/\left(  \gamma_{1}+\gamma
_{2}\right)  .$ The integer sequence $k=k_{n}$ represents the number of top
order statistics satisfying $k\rightarrow\infty$ and $k/n\rightarrow0$ as
$n\rightarrow\infty.$ The sequence of rv's $Z_{1:n}\leq...\leq Z_{n:n}$
represent the order statistics pertaining to the sample $Z_{1},...,Z_{n}$ and
$\delta_{\left[  1:n\right]  },...,\delta_{\left[  n:n\right]  }$ denote the
corresponding concomitant values, satisfying $\delta_{\left[  j:n\right]
}=\delta_{i}$ for $i$ such that $Z_{j:n}=Z_{i}.$ Useful Gaussian
approximations to $\widehat{p}_{k},$ $\widehat{\gamma}_{k}^{\left(  H\right)
}$ and $\widehat{\gamma}_{1,k}^{\left(  EFG\right)  }$ in terms of sequences
of Brownian bridges are given in \cite{BMN-2015}. On the basis of Kaplan-Meier
integration, \cite{WW2014}, proposed a consistent estimator to $\gamma_{1}$
defined by%
\[
\widehat{\gamma}_{1,k}^{\left(  W\right)  }:=%
{\displaystyle\sum\limits_{i=1}^{k}}
\frac{\delta_{\left[  n-i+1:n\right]  }}{i}\frac{\overline{F}_{n}^{\left(
KM\right)  }\left(  Z_{n-i+1:n}\right)  }{\overline{F}_{n}^{\left(  KM\right)
}\left(  Z_{n-k:n}\right)  }\log\frac{Z_{n-i+1:n}}{Z_{n-k:n}},
\]
where%
\[
F_{n}^{\left(  KM\right)  }\left(  z\right)  :=1-\prod_{Z_{i:n}\leq z}\left(
1-\frac{1}{n-i+1}\right)  ^{\delta_{\left[  i:n\right]  }},
\]
denotes the popular Kaplan-Meier estimator of cdf $F$ \citep[] []{KM58}. The
simulation study given by \cite{BMV2018} showed that $\widehat{\gamma}%
_{1,k}^{\left(  W\right)  }$ overall quite good bias and mean squared error
(MSE) performance compared with $\widehat{\gamma}_{1,k}^{\left(  EFG\right)
}.$ The asymptotic normality of the estimator is established in \cite{BWW2019}
by considering Hall's model \citep[] []{Hall82} which is a sub-class of cdf's
satisfying the second-order condition of regular variation $\left(
\ref{second-orderF}\right)  .$ The bias reduction in tail index estimation for
right censored data is addressed by \cite{BBWG2016}, \cite{BMV2018} and
\cite{BWW2019}. From our perspective, the asymptotic analysis of the
Kaplan-Meier based tail functional is rather intricate and technically
demanding. To address this, a recent study by \cite{MNS2025} employed
Nelson-Aalen integral approach to propose a new estimator for the tail index
$\gamma_{1},$ shown to be asymptotically equivalent to $\widehat{\gamma}%
_{1}^{\left(  W\right)  }$ and defined by%
\[
\widehat{\gamma}_{1,k}^{\left(  MNS\right)  }:=\sum_{i=1}^{k}\frac
{\delta_{\left[  n-i+1:n\right]  }}{i}\frac{\overline{F}_{n}^{\left(
NA\right)  }\left(  Z_{n-i+1:n}\right)  }{\overline{F}_{n}^{\left(  NA\right)
}\left(  Z_{n-k:n}\right)  }\log\frac{Z_{n-i+1:n}}{Z_{n-k:n}},
\]
where%
\[
F_{n}^{\left(  NA\right)  }\left(  z\right)  =1-\prod_{Z_{i:n}\leq z}%
\exp\left(  -\frac{1}{n-i+1}\right)  ^{\delta_{\left[  i:n\right]  }},
\]
denotes the well-known Nelson-Aalen estimator of cdf $F$
\citep[] []{Nelson1972}. Moreover, the authors introduced a Nelson-Aalen tail
empirical process and provided a useful Gaussian approximation for it. Using a
functional representation based on this tail process, they established both
the consistency and asymptotic normality of the estimator $\widehat{\gamma
}_{1,k}^{\left(  MNS\right)  }.$ This approach significantly simplifies the
derivation of the asymptotic behavior of extreme value statistics under
censoring. The authors also demonstrated that, as expected, the asymptotic
bias and variance of $\widehat{\gamma}_{1,k}^{\left(  W\right)  }$ and
$\widehat{\gamma}_{1,k}^{\left(  MNS\right)  }$ are identical. Furthermore,
they observed that both tail index estimators exhibit comparable performance
in terms of bias and MSE. In fact, comparative studies between the
Kaplan-Meier and Nelson-Aalen estimators have concluded that display nearly
identical statistical behavior, see, for instance, \cite{Colosimo2002}.
However, both consistency and asymptotic normality of $\widehat{\gamma}%
_{1,k}^{\left(  W\right)  }$ and $\widehat{\gamma}_{1,k}^{\left(  MNS\right)
}$ are contingent upon the assumption that $p>1/2$ (weak censoring). In
contrast, the adapted Hill estimator $\widehat{\gamma}_{1,k}^{\left(
EFG\right)  }$ retains these asymptotic properties for any $0<p<1.$ To address
this limitation, we propose a modification of the estimator $\widehat{\gamma
}_{1,k}^{\left(  MNS\right)  }$ to incorporate the case $p\leq1/2$ (i.e.
strong censoring) be included too. In Section \ref{sec4}, we analyze two real
datasets: one characterized by weak censoring (insurance data, with an
observed proportion of upper non-censored data $p_{obs}\simeq0.76$ ) and the
other by strong censoring (AIDS data, with $p_{obs}\simeq0.28).$ The issue of
extending tail index estimation to strong censoring scenarios was previously
addressed by \cite{BWW2019} who proposed the following estimator%
\[
\widehat{\gamma}_{1,k}^{\left(  BW\right)  }\left(  \beta\right)
:=\frac{\widehat{T}_{k}\left(  \beta\right)  }{1-\beta\widehat{T}_{k}\left(
\beta\right)  },\text{ for }\gamma_{1}\beta>-1,
\]
where
\[
\widehat{T}_{k}\left(  \beta\right)  :=%
{\displaystyle\sum\limits_{i=2}^{k}}
\frac{\overline{F}_{n}^{\left(  KM\right)  }\left(  Z_{n-i+1:n}\right)
}{\overline{F}_{n}^{\left(  KM\right)  }\left(  Z_{n-k:n}\right)  }\left(
\kappa_{-\beta}\left(  \frac{Z_{n-i+1:n}}{Z_{n-k:n}}\right)  -\kappa_{-\beta
}\left(  \frac{Z_{n-i:n}}{Z_{n-k:n}}\right)  \right)  ,
\]
with $\kappa_{-\beta}\left(  u\right)  :=\int_{1}^{u}t^{-\beta-1}dt,$ for
$u>1,$ being the Box-Cox transform. Considering Hall's model
\citep[][]{Hall82}, the authors established the asymptotic normality of
$\widehat{\gamma}_{1,k}^{\left(  BW\right)  }\left(  \beta\right)  $ provided
that $p\left(  1+\beta\gamma_{1}\right)  >1/2.$ Although the estimator
exhibits a smaller MSE, its construction and asymptotic behavior depend on the
unknown parameters $p$ and $\gamma_{1}.$ As such, selecting an appropriate
value for the tuning parameter $\beta$ may not be straightforward for
practitioners. Therefore, it is more practical to propose an alternative to
$\widehat{\gamma}_{1,k}^{\left(  BW\right)  }\left(  \beta\right)  $ which
does not require any parametrization of $\beta.$ In what follows, we introduce
a new estimation method for $\gamma_{1}$ that addresses this concern. Let us
first point out that the construction of the estimator $\widehat{\gamma}%
_{1,k}^{\left(  MNS\right)  }$ is rooted in the following asymptotic
identity:
\begin{equation}
\mathbf{E}\left[  \log\left(  X/t\right)  \mid X>t\right]  :=\int_{t}^{\infty
}\log\frac{y}{t}d\frac{F\left(  y\right)  }{\overline{F}\left(  t\right)
}\rightarrow\gamma_{1},\text{ as }t\rightarrow\infty, \label{ass1}%
\end{equation}
which is equivalent to%
\[
\int_{1}^{\infty}y^{-1}\frac{\overline{F}\left(  ty\right)  }{\overline
{F}\left(  t\right)  }dy\rightarrow\gamma_{1},\text{ as }t\rightarrow\infty.
\]
To precisely identify the stage at which the condition $p>1/2$ becomes
necessary in the analysis of the asymptotic behavior of the proposed
estimator, we carefully examined the detailed steps in the proof of the
Gaussian approximation to the empirical tail process introduced by the
authors. We ultimately concluded that the integrand function $\log\left(
y/t\right)  ,$ in $\left(  \ref{ass1}\right)  ,$ must be multiplied by the
weight $\left(  \overline{F}\left(  y\right)  /\overline{F}\left(  t\right)
\right)  ^{\beta/p-1}$ for some suitable constant $\beta>1,$ so that
\begin{align}
&  \left(  \beta/p\right)  ^{2}\mathbf{E}\left[  \left(  \frac{\overline
{F}\left(  y\right)  }{\overline{F}\left(  t\right)  }\right)  ^{\beta
/p-1}\log\left(  X/t\right)  \mid X>t\right] \label{ass2}\\
&  =\left(  \beta/p\right)  ^{2}\int_{t}^{\infty}\left(  \frac{\overline
{F}\left(  y\right)  }{\overline{F}\left(  t\right)  }\right)  ^{\beta
/p-1}\log\frac{y}{t}d\frac{F\left(  y\right)  }{\overline{F}\left(  t\right)
}\rightarrow\gamma_{1},\text{ as }t\rightarrow\infty,
\end{align}
which is equivalent to
\begin{equation}
\left(  \beta/p\right)  \int_{1}^{\infty}y^{-1}\left(  \frac{\overline
{F}\left(  ty\right)  }{\overline{F}\left(  t\right)  }\right)  ^{\beta
/p}dy\rightarrow\gamma_{1},\text{ as }t\rightarrow\infty, \label{ass2bis}%
\end{equation}
see Proposition \ref{prop2} given in Section \ref{Appendix A}. It is clear
that both statements $\left(  \ref{ass1}\right)  $ and $\left(  \ref{ass2}%
\right)  $ coincide when the tuning parameter $\beta$ is equal to $p.$ Next,
we observe that assertion $\left(  \ref{ass2}\right)  $ indeed yields a
Hill-type estimator to $\gamma_{1}$ which is applicable for any $p\in\left(
0,1\right)  $ provided that the tuning parameter $\beta$ be greater than $1,$
see Section \ref{sec2} for further details. Now, we proceed to define the
weighted Nelson-Aalen integral estimator of the tail index $\gamma_{1}.$ In
$\left(  \ref{ass2}\right)  ,$ letting $t=Z_{n-k:n}$ \ and substituting $F$
and $p$ by $F_{n}^{\left(  NA\right)  }$ and $\widehat{p}_{k}$ respectively
yield the following estimator to $\gamma_{1}:$%
\begin{equation}
\widehat{\gamma}_{1,k}^{\left(  NA\right)  }\left(  \beta\right)  :=\left(
\beta/\widehat{p}_{k}\right)  ^{2}\int_{Z_{n-k:n}}^{\infty}\left(
\frac{\overline{F}_{n}^{\left(  NA\right)  }\left(  y\right)  }{\overline
{F}_{n}^{\left(  NA\right)  }\left(  Z_{n-k:n}\right)  }\right)
^{\beta/\widehat{p}_{k}-1}\log\frac{y}{Z_{n-k:n}}d\frac{F_{n}^{\left(
NA\right)  }\left(  y\right)  }{\overline{F}_{n}^{\left(  NA\right)  }\left(
Z_{n-k:n}\right)  }. \label{f-form}%
\end{equation}
Let us now express the previous integral in a more tractable form, as a sum of
terms derived from the $Z$-sample. To this end, we use the following crucial
representation of cdf $F$ in terms of estimable functions $H$ and $H^{\left(
1\right)  }:$%
\[
\int_{0}^{z}\frac{dH^{\left(  1\right)  }\left(  y\right)  }{\overline
{H}\left(  y-\right)  }=\int_{0}^{z}\frac{dF\left(  y\right)  }{\overline
{F}\left(  y\right)  }=-\log\overline{F}\left(  z\right)  ,\text{ for }z>0,
\]
where $H^{\left(  1\right)  }\left(  z\right)  :=\mathbf{P}\left(  Z\leq
z,\delta=1\right)  =\int_{0}^{z}\overline{G}\left(  y\right)  dF\left(
y\right)  $ and $\overline{H}\left(  y-\right)  :=\lim_{\epsilon\downarrow0}$
$\overline{H}\left(  y-\epsilon\right)  .$ This implies that%
\begin{equation}
\overline{F}\left(  z\right)  =\exp\left\{  -\int_{0}^{z}\frac{dH^{\left(
1\right)  }\left(  y\right)  }{\overline{H}\left(  y-\right)  }\right\}
,\text{ for }z>0. \label{F-formula}%
\end{equation}
The empirical counterparts of cdf $H$ and sub-distribution $H^{\left(
1\right)  }$ are given by
\[
H_{n}\left(  z\right)  :=n^{-1}\sum_{i=1}^{n}\mathbb{I}_{\left\{  Z_{i:n}\leq
z\right\}  }\text{ and }H_{n}^{\left(  1\right)  }\left(  z\right)
:=n^{-1}\sum_{i=1}^{n}\delta_{\left[  i:n\right]  }\mathbb{I}_{\left\{
Z_{i:n}\leq z\right\}  },
\]
see, for instance, \cite{SW86} page 294. Note that $\overline{H}_{n}\left(
Z_{n:n}\right)  =0,$ for this reason and to avoid dividing by zero, we use
$\overline{H}\left(  x-\right)  $ instead of $\overline{H}\left(  x\right)  $
in formula $\left(  \ref{F-formula}\right)  .$ Nelson-Aalen estimator
$F_{n}^{\left(  NA\right)  }\left(  z\right)  $ is obtained by substituting
both $H_{n}$ and $H_{n}^{\left(  1\right)  }$ for $H$ and $H^{\left(
1\right)  }$ respectively in the right-hand side of formula $\left(
\ref{F-formula}\right)  .$ That is, we have
\begin{equation}
\overline{F}_{n}^{\left(  NA\right)  }\left(  z\right)  :=\exp\left\{
-\int_{0}^{z}\frac{dH_{n}^{\left(  1\right)  }\left(  y\right)  }{\overline
{H}_{n}\left(  y-\right)  }\right\}  . \label{emp-F-formula}%
\end{equation}
Note that, since $\overline{H}_{n}\left(  Z_{i:n}-\right)  =\overline{H}%
_{n}\left(  Z_{i-1:n}\right)  =\left(  n-i+1\right)  /n,$ then the previous
integral leads to the expression of $F_{n}^{\left(  NA\right)  }\left(
z\right)  $ given above. Differentiating both sides of $\left(
\ref{emp-F-formula}\right)  $ yields that%
\[
dF_{n}^{\left(  NA\right)  }\left(  z\right)  =\frac{\overline{F}_{n}^{\left(
NA\right)  }\left(  z\right)  }{\overline{H}_{n}\left(  z-\right)  }%
dH_{n}^{\left(  1\right)  }\left(  z\right)  .
\]
By using this in $\left(  \ref{f-form}\right)  ,$ we end up with%
\[
\widehat{\gamma}_{1,k}^{\left(  NA\right)  }\left(  \beta\right)  =\left(
\beta/\widehat{p}_{k}\right)  ^{2}\sum_{i=1}^{k}\frac{\delta_{\left[
n-i+1:n\right]  }}{i}\left(  \frac{\overline{F}_{n}^{\left(  NA\right)
}\left(  Z_{n-i+1:n}\right)  }{\overline{F}_{n}^{\left(  NA\right)  }\left(
Z_{n-k:n}\right)  }\right)  ^{\beta/\widehat{p}_{k}}\log\frac{Z_{n-i+1:n}%
}{Z_{n-k:n}}.
\]
It is worth mentioning that $\widehat{\gamma}_{1,k}^{\left(  NA\right)
}\left(  \beta\right)  $ works even for $0<\beta<1$ but for $p>1/2.$\medskip

\noindent On replacing $\overline{F}_{n}^{\left(  NA\right)  }$ by
$\overline{F}_{n}^{\left(  KM\right)  },$ we define the weighted Kaplan-Meier
integral estimator of the tail index $\gamma_{1}:$%
\[
\widehat{\gamma}_{1,k}^{\left(  KM\right)  }\left(  \beta\right)  :=\left(
\beta/\widehat{p}_{k}\right)  ^{2}\sum_{i=1}^{k}\frac{\delta_{\left[
n-i+1:n\right]  }}{i}\left(  \frac{\overline{F}_{n}^{\left(  KM\right)
}\left(  Z_{n-i+1:n}\right)  }{\overline{F}_{n}^{\left(  KM\right)  }\left(
Z_{n-k:n}\right)  }\right)  ^{\beta/\widehat{p}_{k}}\log\frac{Z_{n-i+1:n}%
}{Z_{n-k:n}},
\]
which meets Worms's estimator $\widehat{\gamma}_{1,k}^{\left(  W\right)  }$
for $\beta=\widehat{p}_{k}.$ To derive more tractable expressions for
$\widehat{\gamma}_{1,k}^{\left(  NA\right)  }\left(  \beta\right)  $ and
$\widehat{\gamma}_{1,k}^{\left(  KM\right)  }\left(  \beta\right)  $ solely in
terms of the pairs $\left(  Z_{j},\delta_{j}\right)  ,$ we make use of the
following two expressions:%
\[
\frac{\overline{F}_{n}^{\left(  NA\right)  }\left(  Z_{n-i+1:n}\right)
}{\overline{F}_{n}^{\left(  NA\right)  }\left(  Z_{n-k:n}\right)  }%
=\prod_{j=i+1}^{k}\exp\left(  -\frac{\delta_{\left[  n-j+1:n\right]  }}%
{j}\right)  ,
\]
and%
\[
\frac{\overline{F}_{n}^{\left(  KM\right)  }\left(  Z_{n-i+1:n}\right)
}{\overline{F}_{n}^{\left(  KM\right)  }\left(  Z_{n-k:n}\right)  }%
=\prod_{j=i+1}^{k}\left(  1-\frac{\delta_{\left[  n-j+1:n\right]  }}%
{j}\right)  .
\]
Thereby, we get%
\[
\widehat{\gamma}_{1,k}^{\left(  NA\right)  }\left(  \beta\right)  =\left(
\beta/\widehat{p}_{k}\right)  ^{2}\sum_{i=1}^{k}\frac{\delta_{\left[
n-i+1:n\right]  }}{i}\prod_{j=i+1}^{k}\exp\left(  -\frac{\beta}{\widehat{p}%
_{k}}\frac{\delta_{\left[  n-j+1:n\right]  }}{j}\right)  \log\frac
{Z_{n-i+1:n}}{Z_{n-k:n}},
\]
and%
\[
\widehat{\gamma}_{1,k}^{\left(  KM\right)  }\left(  \beta\right)  =\left(
\beta/\widehat{p}_{k}\right)  ^{2}\sum_{i=1}^{k}\frac{\delta_{\left[
n-i+1:n\right]  }}{i}\prod_{j=i+1}^{k}\left(  1-\frac{\delta_{\left[
n-j+1:n\right]  }}{j}\right)  ^{\beta/\widehat{p}_{k}}\log\frac{Z_{n-i+1:n}%
}{Z_{n-k:n}}.
\]
In Proposition \ref{prop7} of Section \ref{Appendix A}, we establish a useful
result concerning the connection between $\overline{F}_{n}^{\left(  KM\right)
}$ and $\overline{F}_{n}^{\left(  NA\right)  },$ in the sense that:%
\[
\sup_{0<x\leq Z_{n:n}}\left\vert \frac{\overline{F}_{n}^{\left(  KM\right)
}\left(  x\right)  }{\overline{F}_{n}^{\left(  NA\right)  }\left(  x\right)
}-1\right\vert =O_{\mathbf{P}}\left(  n^{-1}\right)  ,\text{ as }%
n\rightarrow\infty,
\]
which implies that, for every $\beta>0,$ we have%
\[
\widehat{\gamma}_{1,k}^{\left(  KM\right)  }\left(  \beta\right)  =\left(
1+O_{\mathbf{P}}\left(  n^{-1}\right)  \right)  \widehat{\gamma}%
_{1,k}^{\left(  NA\right)  }\left(  \beta\right)  ,\text{ as }n\rightarrow
\infty.
\]
This, amongst others, allows the derivation of the limit theorems of
$\widehat{\gamma}_{1,k}^{\left(  KM\right)  }\left(  \beta\right)  $ from
those of $\widehat{\gamma}_{1,k}^{\left(  NA\right)  }\left(  \beta\right)  ,$
and vice versa, without requiring the assumption of Hall's model. To
streamline the exposition, we use the generic notation:%
\[
\widehat{\gamma}_{1,k}^{\left(  \bullet\right)  }\left(  \beta\right)
:=\left\{  \widehat{\gamma}_{1,k}^{\left(  NA\right)  }\left(  \beta\right)
\text{ or }\widehat{\gamma}_{1,k}^{\left(  KM\right)  }\left(  \beta\right)
\right\}  .
\]
Assuming the first- and second-order conditions of regular variation, rather
than relying on Hall's model, we establish in Section \ref{sec2}, the
consistency and asymptotic normality of proposed estimator $\widehat{\gamma
}_{1,k}^{\left(  \bullet\right)  }\left(  \beta\right)  $ for every $0<p<1.$
Furthermore, in Section \ref{sec3}, we show that $\widehat{\gamma}%
_{1,k}^{\left(  \bullet\right)  }\left(  \beta\right)  $ outperforms the
adapted Hill-type estimator $\widehat{\gamma}_{1,k}^{\left(  EFG\right)  }$ in
both weak and strong censoring scenarios, particularly in terms of bias and
MSE, provided that the tuning parameter $\beta$ is selected slightly greater
than one, namely $\beta=1.01$ or $\beta=1.001$ for instance.$\medskip$

\noindent The remainder of the paper is organized as follows. In Section
\ref{sec2}, we present our main results namely the consistency and asymptotic
normality of the proposed estimator, whose proofs are postponed to Section
\ref{sec5}. In Section \ref{sec3} illustrates, through a simulation study, the
finite sample behavior of our estimator with a comparison with already
existing ones. In Section \ref{sec4}, we apply our results to real-life
datasets of insurance losses and AIDS disease. Finally, some auxiliary results
and the figures related to the simulation study are gathered in Sections
\ref{Appendix A} and \ref{Appendix B} respectively.

\section{\textbf{Main results\label{sec2}}}

\begin{theorem}
\textbf{\label{theorem 1}}Assume that cdf's $F$ and $G$ satisfy the
first-order condition of regular variation $\left(  \ref{RV}\right)  ,$ and
let $k=k_{n}$ be a sequence of integers such that $k\rightarrow\infty$ and
$k/n\rightarrow0.$ Then for every $\beta>1,$ we have%
\begin{equation}
\widehat{\gamma}_{1,k}^{\left(  \bullet\right)  }\left(  \beta\right)
\overset{\mathbf{P}}{\rightarrow}\gamma_{1},\text{ as }n\rightarrow\infty.
\label{consist}%
\end{equation}

\end{theorem}

\noindent Since weak approximations of statistics in extreme value theory are
typically obtained under a second-order framework
\citep[see, e.g., page 48  in] []{deHF06}, then it is natural to assume that
cdf $F$ satisfies the well-known second-order condition of regular variation.
That is, we assume that, for any $x>0,$ we have%
\[
\lim_{t\rightarrow\infty}\frac{U_{F}\left(  tx\right)  /U_{F}\left(  t\right)
-x^{\gamma_{1}}}{A_{1}^{\ast}\left(  t\right)  }=x^{\gamma_{1}}\dfrac
{x^{\tau_{1}}-1}{\tau_{1}},
\]
or equivalently%
\begin{equation}
\lim_{t\rightarrow\infty}\frac{\overline{F}\left(  tx\right)  /\overline
{F}\left(  t\right)  -x^{-1/\gamma_{1}}}{A_{1}\left(  t\right)  }%
=x^{-1/\gamma_{1}}\frac{x^{\tau_{1}/\gamma_{1}}-1}{\tau_{1}\gamma_{1}},
\label{second-orderF}%
\end{equation}
where $\tau_{1}\leq0$\ is the second-order parameter,and $A_{1}\left(
t\right)  :=A_{1}^{\ast}\left(  1/\overline{F}\left(  t\right)  \right)  $
with $A_{1}^{\ast}$ being a function tending to $0,$ not changing sign near
infinity and having a regularly varying absolute value at infinity with index
$\tau_{1}.$\ If $\tau_{1}=0,$ interpret $\left(  x^{\tau_{1}}-1\right)
/\tau_{1}$ as $\log x.$ The functions $L^{\leftarrow}\left(  s\right)
:=\inf\left\{  x:L\left(  x\right)  \geq s\right\}  ,$ for $0<s<1,$ and
$U_{L}\left(  t\right)  :=L^{\leftarrow}\left(  1-1/t\right)  ,$ $t>1,$
respectively stand for the quantile and tail quantile functions of a cdf $L.$
This condition is satisfied by a broad class of commonly used distributions,
including Burr, F\'{e}chet, GPD, GEV and Student distributions, all of which
fall under Hall's model. However, the modified Pareto distribution (with a
logarithmic term) $F\left(  x\right)  =1-x^{-1/\gamma_{1}}\left(  1+1/\log
x\right)  ,$ $x>e$ satisfies $\left(  \ref{second-orderF}\right)  ,$ with
$A_{1}\left(  t\right)  =1/\log t,$ yet it does not belong to Hall's class.

\begin{theorem}
\textbf{\label{theorem 2}}In addition to the assumptions of Theorem
$\ref{theorem 1},$ suppose that the second-order condition of regular
variation $\left(  \ref{second-orderF}\right)  $ holds and let $h:=U_{H}%
\left(  n/k\right)  .$ If $\sqrt{k}A_{1}\left(  h\right)  =O\left(  1\right)
,$ then on the probability space $\left(  \Omega,\mathcal{A},\mathbf{P}%
\right)  ,$ there exists a sequence of standard Wiener processes $\left\{
W_{n}\left(  s\right)  ;\text{ }0<s<1\right\}  _{n\geq1}$ such that, for every
$\beta>1,$ we have%
\[
\sqrt{k}\left(  \widehat{\gamma}_{1,k}^{\left(  \bullet\right)  }\left(
\beta\right)  -\gamma_{1}\right)  =\gamma_{1}N_{n}\left(  \beta\right)
+\frac{\beta}{\beta-p\tau_{1}}\sqrt{k}A_{1}\left(  h\right)  +o_{\mathbf{P}%
}\left(  1\right)  ,\text{ as }n\rightarrow\infty,
\]
where%
\begin{equation}
N_{n}\left(  \beta\right)  :=\frac{\beta\left(  p+\beta-1\right)  }{p}\int%
_{0}^{1}s^{\beta-2}\mathbf{W}_{n,1}\left(  s\right)  ds+\frac{\beta}%
{p}\mathbf{W}_{n,1}\left(  1\right)  +\beta\int_{0}^{1}s^{\beta-2}%
\mathbf{W}_{n,2}\left(  s\right)  ds, \label{Gaussian-app}%
\end{equation}
is a sequence of centred Gaussian rv's, with%
\[
\mathbf{W}_{n,1}\left(  s\right)  :=W_{n}\left(  \theta\right)  -W_{n}\left(
\theta-ps\right)  ,\text{ for }0<s\leq\theta/p\text{,}%
\]
and%
\[
\mathbf{W}_{n,2}\left(  s\right)  :=W_{n}\left(  1\right)  -W_{n}\left(
1-qs\right)  ,\text{ for }0\leq s\leq1,
\]
being two independent centred Gaussian processes, with$\ \theta:=H^{\left(
1\right)  }\left(  \infty\right)  $ and $q:=1-p.$ Whenever $\sqrt{k}%
A_{1}\left(  h\right)  \rightarrow\lambda<\infty,$ then%
\begin{equation}
\sqrt{k}\left(  \widehat{\gamma}_{1,k}^{\left(  \bullet\right)  }\left(
\beta\right)  -\gamma_{1}\right)  \overset{\mathcal{D}}{\rightarrow
}\mathcal{N}\left(  \mu_{\beta},\sigma_{\beta}^{2}\right)  ,\text{ as
}n\rightarrow\infty, \label{normality}%
\end{equation}
where%
\[
\mu_{\beta}:=\frac{\lambda\beta}{\beta-p\tau_{1}}\text{ and }\sigma_{\beta
}^{2}:=\frac{1}{p}\frac{\beta^{2}}{2\beta-1}\gamma_{1}^{2}.
\]

\end{theorem}

\begin{remark}
Theorem $\ref{theorem 1}$ remains valid if one uses any other consistent
estimator $\overline{p}_{k}$ for $p$ instead of $\widehat{p}_{k}.$ Moreover,
in view of the statement $\left(  \ref{final-app}\right)  ,$ Theorem
$\ref{theorem 2}$ also remains valid provided that $\sqrt{k}\left(
\overline{p}_{k}-p\right)  $ is asymptotically bounded. In other words, the
choice of the estimator of $p$ has no effect on the asymptotic variance
$\sigma_{\beta}^{2}.$
\end{remark}

\begin{remark}
For $\beta=p>1/2,$ i.e. when $\widehat{\gamma}_{1,k}^{\left(  NA\right)
}\left(  \beta\right)  $ meets $\widehat{\gamma}_{1,k}^{\left(  MNS\right)
},$ the asymptotic variance equals $\gamma_{1}^{2}p/\left(  2p-1\right)  $
which corresponds to that of Theorem 3.1 in \citep[] []{MNS2025}.
\end{remark}

\begin{remark}
If $p>1/2,$ then $\sigma_{\beta}^{2}<\sigma_{p}^{2}.$ This means that, in the
weak censoring case, the asymptotic variance of $\widehat{\gamma}%
_{1,k}^{\left(  \bullet\right)  }\left(  \beta\right)  $ decreases compared to
that of $\widehat{\gamma}_{1,k}^{\left(  MNS\right)  }.$ Indeed, we have%
\[
\sigma_{\beta}^{2}-\sigma_{p}^{2}=\frac{\left(  \beta-p\right)  \left(
2p\beta-p-\beta\right)  }{p\left(  2p-1\right)  \left(  2\beta-1\right)
}\gamma_{1}^{2}<0,
\]
provided that $1<\beta<p/\left(  2p-1\right)  ,$ which holds because
$p/\left(  2p-1\right)  >1.$ One may choose $\epsilon,$ sufficiently small, so
that $p/\left(  2p-1\right)  >1+\epsilon=\beta.$
\end{remark}

\begin{remark}
According to \cite{BWW2019}, the asymptotic variance of $\widehat{\gamma
}_{1,k}^{\left(  BW\right)  }\left(  \beta\right)  $ equals $\gamma_{1}%
^{2}p\left(  1+\beta\gamma_{1}\right)  ^{2}/\left(  2p\left(  1+\beta
\gamma_{1}\right)  -1\right)  ,$ whereas that of $\widehat{\gamma}%
_{1,k}^{\left(  W\right)  }$ is \ equal to $\gamma_{1}^{2}p/\left(
2p-1\right)  ,$ for $p>1/2.$ However, the authors did not provide any
discussion regarding the relative magnitude of the first variance compared to
the second. A simple algebraic manipulation shows that the difference between
the two variances is equal to%
\[
\frac{p\beta\gamma_{1}}{2p-1}\frac{\left(  2p-1\right)  \beta\gamma
_{1}-2\left(  1-p\right)  }{2p\left(  1+\beta\gamma_{1}\right)  -1}.
\]
For $p>1/2$ and $\gamma_{1}\beta>-1,$ the sign of both ratios is not
immediately clear, as it depends on the interaction among $p,$ $\gamma_{1}$
and $\beta.$ Even if $\beta>0,$ the first ratio and the denominator of the
second ratio are positive, the sign of the numerator in the second ratio may
still be ambiguous. In summary, for $\sigma_{\gamma_{1},\beta}^{2}$ to be
smaller than $\sigma_{\gamma_{1}}^{2},$ in the weak censoring case, the
following condition must hold: $0<\gamma_{1}\beta<2\left(  1-p\right)
/\left(  2p-1\right)  .$ In conclusion, choosing a suitable value for $\beta$
that ensures this inequality is not straightforward from a practical standpoint.
\end{remark}

\begin{remark}
As previously mentioned, Theorem \ref{second-orderF} is formulated to
simultaneously generalize the asymptotic normality of both estimators
$\widehat{\gamma}_{1,k}^{\left(  MNS\right)  }$ and $\widehat{\gamma}%
_{1,k}^{\left(  W\right)  }$ under the condition $p>1/2.$ Furthermore, it
relaxes the assumptions required for the asymptotic normality of
$\widehat{\gamma}_{1,k}^{\left(  W\right)  }$\citep[] []{BWW2019} by extending
its applicability to a broader class of regularly varying functions, beyond
the classical Hall-type models. In addition, the theorem offers a useful
functional representation of the limiting distribution in terms of a sequence
of Wiener processes.
\end{remark}

\section{\textbf{Simulation study\label{sec3}}}

\noindent First, it is noteworthy that this study is carried out with respect
to the estimator $\widehat{\gamma}_{1,k}^{\left(  NA\right)  }\left(
\beta\right)  .$ To assess the performance of $\widehat{\gamma}_{1,k}^{\left(
NA\right)  }\left(  \beta\right)  $ and compare it with $\widehat{\gamma
}_{1,k}^{\left(  MNS\right)  }$ and $\widehat{\gamma}_{1,k}^{\left(
EFG\right)  },$ we use two sets of data from Burr model with parameters
$(\zeta,\eta)$ and Fr\'{e}chet model with parameter $\zeta$,\ respectively
defined by
\[
1-\left(  1+x^{1/\eta}\right)  ^{-\eta/\zeta}\text{\quad and\quad}%
\exp(-x^{-1/\zeta}),\text{ }x>0,
\]
and zero elsewhere.\medskip

\noindent We consider two censoring scenarios as follows:

\begin{itemize}
\item[$\bullet$] Burr $(\gamma_{1},\eta)$ censored by Burr $(\gamma_{2}%
,\eta).$

\item[$\bullet$] Fr\'{e}chet $(\gamma_{1})$ censored by Fr\'{e}chet
$(\gamma_{2}).$
\end{itemize}

\noindent We fix $\eta=$ $0.25$ and consider three values for the tuning
parameter $\beta=1.01,1.5$ and $2,$ two values for the tail index $\gamma
_{1}=0.4$ and $0.7,$ and three censoring percentages $1-p=0.70,0.50$ and
$0.30.$ For each pair $\left(  \gamma_{1},p\right)  ,$ we solve the equation
$p=\gamma_{2}/\left(  \gamma_{1}+\gamma_{2}\right)  $\ to determine the
corresponding $\gamma_{2}$-value. We generate $2000$ samples of size $1000$
for each scenario and present the (averaged) results in Figures \ref{fig1}%
-\ref{fig6} in Section \ref{Appendix B}.\medskip

\noindent\textbf{Discussion of the simulation results}\medskip

\noindent The results demonstrate the superiority of $\widehat{\gamma}%
_{1,k}^{\left(  NA\right)  }\left(  \beta\right)  $ over both estimators
$\widehat{\gamma}_{1,k}^{\left(  MNS\right)  }$ and $\widehat{\gamma}%
_{1,k}^{\left(  EFG\right)  }$in terms of absolute bias (Bias) and MSE under
several specific conditions. Our estimator $\widehat{\gamma}_{1,k}^{\left(
NA\right)  }\left(  \beta\right)  $ shows better performance for large values
of the tail index, as its Bias and MSE are lower with respect to those of
$\widehat{\gamma}_{1,k}^{\left(  MNS\right)  }$ and $\widehat{\gamma}%
_{1,k}^{\left(  EFG\right)  }.$ This advantage also appears clearly when the
proportion of uncensored data is small, i.e., in the presence of high
censorship. Furthermore, $\widehat{\gamma}_{1,k}^{\left(  NA\right)  }\left(
\beta\right)  $ performs well when the tuning parameter $\beta$ is close to
one (i.e., for $\beta=1.01),$ where the results are more stable and accurate.
However, as $\beta$ increases, the efficiency of the estimator decreases, as
reflected in the rise of both Bias and MSE. In such cases, $\widehat{\gamma
}_{1,k}^{\left(  NA\right)  }\left(  \beta\right)  $ becomes less competitive
compared to the other two estimators.\medskip

\noindent Overall, our estimator exhibits promising performance in specific
settings, especially in cases of high censoring rate or heavier tails.
However, its sensitivity to changes in the value of the tuning parameter
should be taken into account.

\section{\textbf{Real data application\label{sec4}}}

\subsection{Insurance loss data}

\noindent The insurance loss dataset, collected by the US Insurance Services
Office, Inc., is publicly accessible via the copula package in the R
statistical software. This dataset has been the subject of multiple analyses,
including those by \cite{FV98}, \cite{Klugman99} and \cite{DPV06}. It
comprises $1500$ observations, among which $34$ are right-censored, indicating
that the corresponding losses exceed policy limits, which vary from contract
to contract. Using the adaptive selection algorithm proposed by Reiss and
Thomas (see \citep[][]{ReTo97}, page 121), the optimal sample fraction
$k_{opt}$ for estimating $p$ is found to be $51,$ yielding an estimate
$\widehat{p}_{k}=0.76.$ For the adapted Hill estimator, the optimal sample
fraction is $73,$ leading to an estimate of $\widehat{\gamma}_{1,k}^{\left(
EFG\right)  }=0.77.$ Similarly, the optimal sample fraction for
$\widehat{\gamma}_{1,k}^{\left(  MNS\right)  }$ is $30,$ resulting in
$\widehat{\gamma}_{1,k}^{\left(  MNS\right)  }=0.45.$ The algorithm used to
determine the optimal number $k_{opt},$ of upper order statistics, essential
for computing each estimator, is defined by%
\[
k_{opt}:=\arg\min_{1<k<n}\frac{1}{k}\sum_{i=1}^{k}i^{\nu}\left\vert
\widehat{\xi}_{i}-\text{median}\left\{  \widehat{\xi}_{1},...,\widehat{\xi
}_{k}\right\}  \right\vert ,\text{ }0\leq\nu\leq1/2,
\]
where $\widehat{\xi}_{i}$ is $\xi$-estimator, namely $(\widehat{p}%
_{k},\widehat{\gamma}_{1,k}^{\left(  EFG\right)  },\widehat{\gamma}%
_{1}^{\left(  MNS\right)  }$ or $\widehat{\gamma}_{1,k}^{\left(  NA\right)
}\left(  \beta\right)  )$ based on the $i$ upper order statistics. Following
the recommendation of \cite{NF2004}, we choose $\nu=0.3$ which has been shown
to yield improved performance in terms of bias and MSE. As discussed in
Section \ref{sec3}, the proposed estimator performs best when the tuning
parameter is set to $\beta=1.01.$ Under this choice, the optimal number of
upper order statistics is also $30$, resulting in a tail index estimate of
$\widehat{\gamma}_{1,k}^{\left(  NA\right)  }\left(  \beta\right)  =0.51.$

\subsection{Australian AIDS data}

\noindent The Australian AIDS survival dataset, compiled by Dr. P. J. Solomon
and the Australian National Centre in HIV Epidemiology and Clinical Research,
comprises records of $2843$ diagnosed prior to July $1,1991.$ For the purposes
of this analysis, we focus on a subset of $2754$ male patients. This dataset
has been extensively analyzed in the literature, see for instance \cite{RS-94}
and \cite{VR-02} (pages 379-385), \cite{EnFG08} and \cite{Ndao16}. Adopting
the same methodology as in the previous example, we determine an optimal
sample fraction of $162$ for the estimation of $p,$ yielding $\widehat{p}%
_{k}=0.30.$ For the adapted Hill estimator, the optimal number of upper order
statistics is $55,$ resulting in $\widehat{\gamma}_{1,k}^{\left(  EFG\right)
}=0.72.$ Interestingly, the same optimal threshold, $55,$ is identified for
$\widehat{\gamma}_{1,k}^{\left(  MNS\right)  },$ but with a notably lower
estimate $\widehat{\gamma}_{1,k}^{\left(  MNS\right)  }=0.15.$ When applying
our proposed estimator with tuning parameter $\beta=1.01,$ the optimal number
of top order statistics increases to $275,$ leading to an estimate
$\widehat{\gamma}_{1,k}^{\left(  NA\right)  }\left(  \beta\right)  =0.64.$

\section{\textbf{Proofs\label{sec5}}}

\noindent We will focus on proving the consistency and asymptotic normality of
$\widehat{\gamma}_{1,k}^{\left(  NA\right)  }\left(  \beta\right)  $ as the
corresponding results for $\widehat{\gamma}_{1,k}^{\left(  KM\right)  }\left(
\beta\right)  $ can be directly inferred from Proposition \ref{prop7}.

\subsection{\textbf{Proof }of Theorem \ref{theorem 1}}

Note that changing variables $x=y/Z_{n-k:n}$ then integrating by parts in
$\left(  \ref{f-form}\right)  ,$ yield
\[
\widehat{\gamma}_{1,k}^{\left(  NA\right)  }\left(  \beta\right)  =\frac
{\beta}{\widehat{p}_{k}}\int_{1}^{\infty}x^{-1}\left(  \frac{\overline{F}%
_{n}^{\left(  NA\right)  }\left(  Z_{n-k:n}x\right)  }{\overline{F}%
_{n}^{\left(  NA\right)  }\left(  Z_{n-k:n}\right)  }\right)  ^{\beta
/\widehat{p}_{k}}dx.
\]
Let us decompose $\widehat{\gamma}_{1,k}^{\left(  NA\right)  }\left(
\beta\right)  -\gamma_{1}$ into the sum of%
\[
T_{n1}:=\frac{\beta}{p}\int_{1}^{\infty}x^{-1}\left(  \frac{\overline
{F}\left(  Z_{n-k:n}x\right)  }{\overline{F}\left(  Z_{n-k:n}\right)
}\right)  ^{\beta/p}dx-\gamma_{1},
\]%
\[
T_{n2}:=\beta\left(  \frac{1}{\widehat{p}_{k}}-\frac{1}{p}\right)  \int%
_{1}^{\infty}x^{-1}\left(  \frac{\overline{F}_{n}^{\left(  NA\right)  }\left(
Z_{n-k:n}x\right)  }{\overline{F}_{n}^{\left(  NA\right)  }\left(
Z_{n-k:n}\right)  }\right)  ^{\beta/\widehat{p}_{k}}dx,
\]%
\[
T_{n3}:=\frac{\beta}{p}\int_{1}^{\infty}x^{-1}\left\{  \left(  \frac
{\overline{F}_{n}^{\left(  NA\right)  }\left(  Z_{n-k:n}x\right)  }%
{\overline{F}_{n}^{\left(  NA\right)  }\left(  Z_{n-k:n}\right)  }\right)
^{\beta/\widehat{p}_{k}}-\left(  \frac{\overline{F}_{n}^{\left(  NA\right)
}\left(  Z_{n-k:n}x\right)  }{\overline{F}_{n}^{\left(  NA\right)  }\left(
Z_{n-k:n}\right)  }\right)  ^{\beta/p}\right\}  dx,
\]
and%
\[
T_{n4}:=\frac{\beta}{p}\int_{1}^{\infty}x^{-1}\left\{  \left(  \frac
{\overline{F}_{n}^{\left(  NA\right)  }\left(  Z_{n-k:n}x\right)  }%
{\overline{F}_{n}^{\left(  NA\right)  }\left(  Z_{n-k:n}\right)  }\right)
^{\beta/p}-\left(  \frac{\overline{F}\left(  Z_{n-k:n}x\right)  }{\overline
{F}\left(  Z_{n-k:n}\right)  }\right)  ^{\beta/p}\right\}  dx.
\]
Since $Z_{n-k:n}\overset{\mathbf{P}}{\rightarrow}\infty,$ then from assertion
$\left(  i\right)  $ of Proposition \ref{prop2} we infer that $T_{n1}%
\overset{\mathbf{P}}{\rightarrow}0,$ as $n\rightarrow\infty.$ Then, it remains
to show that $T_{ni}=o_{\mathbf{P}}\left(  1\right)  ,$ for $i=2,3,4.$
Assertion $\left(  iii\right)  $ of Proposition \ref{prop3} states that, for
every $1/4<\eta<1/2$ and small $0<\epsilon<1,$ we have uniformly over $x>1$%
\begin{equation}
\frac{\overline{F}_{n}^{\left(  NA\right)  }\left(  Z_{n-k:n}x\right)
}{\overline{F}_{n}^{\left(  NA\right)  }\left(  Z_{n-k:n}\right)
}=x^{-p/\gamma}\left(  1+o_{\mathbf{P}}\left(  x^{2\eta/\gamma-\epsilon
}\right)  \right)  ,\text{ as }n\rightarrow\infty. \label{app-Fn}%
\end{equation}
It follows that%
\begin{align*}
T_{n2}  &  =\beta\left(  \frac{1}{\widehat{p}_{k}}-\frac{1}{p}\right)
\int_{1}^{\infty}x^{-1}\left(  x^{-p/\gamma}\left(  1+o_{\mathbf{P}}\left(
x^{2\eta/\gamma-\epsilon}\right)  \right)  \right)  ^{\beta/\widehat{p}_{k}%
}dx\\
&  =\beta\left(  \frac{1}{\widehat{p}_{k}}-\frac{1}{p}\right)  \int%
_{1}^{\infty}x^{-1-\beta p/\left(  \gamma\widehat{p}_{k}\right)  }\left(
1+o_{\mathbf{P}}\left(  x^{2\eta/\gamma-\epsilon}\right)  \right)  dx.
\end{align*}
The latter integral equals $\left(  1+o_{\mathbf{P}}\left(  1\right)  \right)
\gamma\widehat{p}_{k}/\left(  \beta p\right)  $ provided that $\beta
>2\widehat{p}_{k}\eta/p.$ Since $\eta>1/4,$ then the previous assumption
becomes $\beta>\widehat{p}_{k}/\left(  2p\right)  .$ Note that $\widehat{p}%
_{k}$ is a consistent estimator for $p$ \citep[see, e.g.,] []{EnFG08}, it
follows that $\widehat{p}_{k}/\left(  2p\right)  \overset{\mathbf{P}%
}{\rightarrow}1/2,$ this means that the previous inequality holds, as
$n\rightarrow\infty,$ because $\beta>1.$ In summary, we showed that
\[
T_{n2}=-\left(  1+o_{\mathbf{P}}\left(  1\right)  \right)  \frac{\gamma}%
{p^{2}}\left(  \widehat{p}_{k}-p\right)  ,
\]
which tends to zero as $n\rightarrow\infty.$ Let us now consider the third
term $T_{n3}.$ By applying the mean value theorem to function $x\rightarrow
a^{x},$ we write for every real $u$ and $v$ and $a>0$%
\[
a^{u}-a^{v}=\exp\left(  u\log a\right)  -\exp\left(  v\log a\right)  =\left(
u-v\right)  a^{w}\log a,
\]
where $w$ is between $u$ and $v.$ Letting $u:=1/\widehat{p},$ $v:=1/p$ and%
\[
a:=\left(  \frac{\overline{F}_{n}^{\left(  NA\right)  }\left(  Z_{n-k:n}%
x\right)  }{\overline{F}_{n}^{\left(  NA\right)  }\left(  Z_{n-k:n}\right)
}\right)  ^{\beta},
\]
yields
\[
T_{n3}=\beta\left(  \frac{1}{\widehat{p}_{k}}-\frac{1}{p}\right)  \int%
_{1}^{\infty}x^{-1}\left(  \frac{\overline{F}_{n}^{\left(  NA\right)  }\left(
Z_{n-k:n}x\right)  }{\overline{F}_{n}^{\left(  NA\right)  }\left(
Z_{n-k:n}\right)  }\right)  ^{\beta/\widetilde{p}_{k}}\log\left(
\frac{\overline{F}_{n}^{\left(  NA\right)  }\left(  Z_{n-k:n}x\right)
}{\overline{F}_{n}^{\left(  NA\right)  }\left(  Z_{n-k:n}\right)  }\right)
^{\beta/p}dx,
\]
where $\widetilde{p}_{k}$ is between $p$ and $\widehat{p}_{k}.$ Once again by
using statement $\left(  \ref{app-Fn}\right)  ,$ the integral just above
becomes%
\begin{align*}
&  \int_{1}^{\infty}x^{-1}\left(  x^{-p/\gamma}\left(  1+o_{\mathbf{P}}\left(
x^{2\eta/\gamma-\epsilon}\right)  \right)  \right)  ^{\beta/\widetilde{p}_{k}%
}\log\left(  x^{-p/\gamma}\left(  1+o_{\mathbf{P}}\left(  x^{2\eta
/\gamma-\epsilon}\right)  \right)  \right)  ^{\beta/p}dx\\
&  =\int_{1}^{\infty}x^{-1-\beta p/\left(  \widetilde{p}_{k}\gamma\right)
}\left(  1+o_{\mathbf{P}}\left(  x^{2\eta/\gamma-\epsilon}\right)  \right)
\log x^{-\beta/\gamma}dx\\
&  +o_{\mathbf{P}}\left(  1\right)  \int_{1}^{\infty}x^{-1-\beta p/\left(
\widetilde{p}_{k}\gamma\right)  +2\eta/\gamma-\epsilon}\left(  1+o_{\mathbf{P}%
}\left(  x^{2\eta/\gamma-\epsilon}\right)  \right)  dx%
\begin{array}
[c]{c}%
=:
\end{array}
I_{1}+I_{2}.
\end{align*}
Since~$\widehat{p}_{k}\overset{\mathbf{P}}{\rightarrow}p$ then $\widetilde{p}%
_{k}\overset{\mathbf{P}}{\rightarrow}p$ as well, thus, for all large $n,$ we
have%
\[
I_{1}=-\frac{\beta}{\gamma}\int_{1}^{\infty}x^{-1-\left(  1+o_{\mathbf{P}%
}\left(  1\right)  \right)  \beta/\gamma}\left(  1+o_{\mathbf{P}}\left(
x^{2\eta/\gamma-\epsilon}\right)  \right)  \left(  \log x\right)  dx,
\]
and
\[
I_{2}=o_{\mathbf{P}}\left(  1\right)  \int_{1}^{\infty}x^{-1-\left(
1+o_{\mathbf{P}}\left(  1\right)  \right)  \beta/\gamma+2\eta/\gamma-\epsilon
}\left(  1+o_{\mathbf{P}}\left(  x^{2\eta/\gamma-\epsilon}\right)  \right)
dx.
\]
It is easy to verify that $I_{1}=-\left(  1+o_{\mathbf{P}}\left(  1\right)
\right)  \gamma/\beta,$ provided that $-\left(  1+o_{\mathbf{P}}\left(
1\right)  \right)  \beta+2\eta<0.$ In other words, with probability tending to
one, $I_{1}$ is finite provided that $\beta>\left(  1+o_{\mathbf{P}}\left(
1\right)  \right)  2\eta,$ which is satisfied for any $\eta>1/4,$ because
$\beta>1.$ By using similar arguments, we show that $I_{2}=o_{\mathbf{P}%
}\left(  1\right)  ,$ for every $\beta>1.$ In summary, we have%
\[
T_{n3}=-\left(  1+o_{\mathbf{P}}\left(  1\right)  \right)  \gamma\left(
\frac{1}{\widehat{p}_{k}}-\frac{1}{p}\right)  =\left(  1+o_{\mathbf{P}}\left(
1\right)  \right)  \frac{\gamma}{p^{2}}\left(  \widehat{p}_{k}-p\right)  ,
\]
which tends to zero in probability as $n\rightarrow\infty.$ It is worth
mentioning that%
\begin{equation}
T_{n2}+T_{n3}=-\left(  1+o_{\mathbf{P}}\left(  1\right)  \right)  \frac
{\gamma}{p^{2}}\left(  \widehat{p}_{k}-p\right)  +\left(  1+o_{\mathbf{P}%
}\left(  1\right)  \right)  \frac{\gamma}{p^{2}}\left(  \widehat{p}%
_{k}-p\right)  =o_{\mathbf{P}}\left(  \widehat{p}_{k}-p\right)  .
\label{Tn1-Tn2}%
\end{equation}
This result is particularly valuable as it eliminates the need for additional
computations in the proof of Theorem \ref{theorem 2} below. Next, we show that
$T_{n4}=o_{\mathbf{P}}\left(  1\right)  $ as well. Once again, invoking the
mean value theorem, we obtain%
\[
\left(  \frac{\overline{F}_{n}^{\left(  NA\right)  }\left(  Z_{n-k:n}x\right)
}{\overline{F}_{n}^{\left(  NA\right)  }\left(  Z_{n-k:n}\right)  }\right)
^{\beta/p}-\left(  \frac{\overline{F}\left(  Z_{n-k:n}x\right)  }{\overline
{F}\left(  Z_{n-k:n}\right)  }\right)  ^{\beta/p}=\frac{\beta}{p}\left(
\eta_{n}\left(  x\right)  \right)  ^{\beta/p-1}L_{n}\left(  x\right)  ,
\]
where%
\[
\min\left\{  \frac{\overline{F}\left(  Z_{n-k:n}x\right)  }{\overline
{F}\left(  Z_{n-k:n}\right)  },\frac{\overline{F}_{n}^{\left(  NA\right)
}\left(  Z_{n-k:n}x\right)  }{\overline{F}_{n}^{\left(  NA\right)  }\left(
Z_{n-k:n}\right)  }\right\}  <\eta_{n}\left(  x\right)  <\max\left\{
\frac{\overline{F}\left(  Z_{n-k:n}x\right)  }{\overline{F}\left(
Z_{n-k:n}\right)  },\frac{\overline{F}_{n}^{\left(  NA\right)  }\left(
Z_{n-k:n}x\right)  }{\overline{F}_{n}^{\left(  NA\right)  }\left(
Z_{n-k:n}\right)  }\right\}
\]
and
\begin{equation}
L_{n}\left(  x\right)  :=\frac{\overline{F}_{n}^{\left(  NA\right)  }\left(
Z_{n-k:n}x\right)  }{\overline{F}_{n}^{\left(  NA\right)  }\left(
Z_{n-k:n}\right)  }-\frac{\overline{F}\left(  Z_{n-k:n}x\right)  }%
{\overline{F}\left(  Z_{n-k:n}\right)  }, \label{Ln}%
\end{equation}
which implies that $T_{n4}=\left(  \beta/p\right)  ^{2}\int_{1}^{\infty}%
x^{-1}\left(  \eta_{n}\left(  x\right)  \right)  ^{\beta/p-1}L_{n}\left(
x\right)  dx.$ We state in Proposition \ref{prop1} that, for every small
$0<\epsilon<1$ and for all large $t,$ we have%
\begin{equation}
\frac{\overline{F}\left(  tx\right)  }{\overline{F}\left(  t\right)
}=x^{-p/\gamma}\left(  1+o\left(  x^{-\epsilon}\right)  \right)  ,\text{ as
}t\rightarrow\infty, \label{appF}%
\end{equation}
uniformly over $x>1.$ Combining both statements $\left(  \ref{app-Fn}\right)
$ and $\left(  \ref{appF}\right)  $ yields that%
\begin{equation}
\eta_{n}\left(  x\right)  =O_{\mathbf{P}}\left(  x^{-p/\gamma}\left(
1+o\left(  x^{2\eta/\gamma-\epsilon}\right)  \right)  \right)  .
\label{appeta}%
\end{equation}
On the other hand, from assertion $\left(  ii\right)  $ of Proposition
$\ref{prop3},$ we infer that, for every $1/4<\eta<1/2$ and for any small
$0<\epsilon<1,$ we have%
\begin{equation}
L_{n}\left(  x\right)  =o_{\mathbf{P}}\left(  x^{\left(  2\eta-p\right)
/\gamma-\epsilon}\right)  ,\text{ } \label{appLn}%
\end{equation}
uniformly over $x>1.$ On the other hand, from approximations $\left(
\ref{appeta}\right)  $ and $\left(  \ref{appLn}\right)  ,$ we write
\begin{align*}
T_{n4}  &  =\left(  \beta/p\right)  ^{2}\int_{1}^{\infty}x^{-1}\left(
x^{-p/\gamma}\left(  1+o_{\mathbf{P}}\left(  x^{2\eta/\gamma-\epsilon}\right)
\right)  \right)  ^{\beta/p-1}o_{\mathbf{P}}\left(  x^{\left(  2\eta-p\right)
/\gamma-\epsilon}\right)  dx\\
&  =o_{\mathbf{P}}\left(  1\right)  \int_{1}^{\infty}x^{-1+\left(  2\eta
-\beta\right)  /\gamma-\epsilon}\left(  1+o_{\mathbf{P}}\left(  x^{2\eta
/\gamma-\epsilon}\right)  \right)  dx.
\end{align*}
The latter integral is finite provided that $\beta>4\eta$ which holds because
$\eta>1/4$ and $\beta>1,$ thus $T_{n4}=o_{\mathbf{P}}\left(  1\right)  .$

\subsection{\textbf{Proof }of Theorem \ref{theorem 2}}

We will consider the same decomposition of $\widehat{\gamma}_{1}^{\left(
NA\right)  }\left(  \beta\right)  -\gamma_{1}$ that we used in the proof
Theorem \ref{theorem 1}. Let us write
\begin{align*}
T_{n1} &  =\left(  \beta/p\right)  \int_{1}^{\infty}x^{-1}\left(
\frac{\overline{F}\left(  Z_{n-k:n}x\right)  }{\overline{F}\left(
Z_{n-k:n}\right)  }\right)  ^{\beta/p}dx-\gamma_{1}\\
&  =\left(  \beta/p\right)  \int_{1}^{\infty}x^{-1}\left\{  \left(
\frac{\overline{F}\left(  Z_{n-k:n}x\right)  }{\overline{F}\left(
Z_{n-k:n}\right)  }\right)  ^{\beta/p}-\left(  x^{-1/\gamma_{1}}\right)
^{\beta/p}\right\}  dx.
\end{align*}
Applying Taylor expansion yields the decomposition of $T_{n1}$ into the sum
of
\[
T_{n1}^{\left(  1\right)  }=\left(  \beta/p\right)  ^{2}\int_{1}^{\infty
}x^{-1}\left(  \frac{\overline{F}\left(  Z_{n-k:n}x\right)  }{\overline
{F}\left(  Z_{n-k:n}\right)  }-x^{-1/\gamma_{1}}\right)  \left(
x^{-1/\gamma_{1}}\right)  ^{\beta/p-1}dx,
\]
and
\[
T_{n1}^{\left(  2\right)  }=\frac{\left(  \beta/p-1\right)  \left(
\beta/p\right)  ^{2}}{2}\int_{1}^{\infty}x^{-1}\left(  \frac{\overline
{F}\left(  Z_{n-k:n}x\right)  }{\overline{F}\left(  Z_{n-k:n}\right)
}-x^{-1/\gamma_{1}}\right)  ^{2}\left(  \rho_{n}\left(  x\right)  \right)
^{\beta/p-2}dx,
\]
where
\[
\min\left\{  \frac{\overline{F}\left(  Z_{n-k:n}x\right)  }{\overline
{F}\left(  Z_{n-k:n}\right)  },x^{-1/\gamma_{1}}\right\}  <\rho_{n}\left(
x\right)  <\max\left\{  \frac{\overline{F}\left(  Z_{n-k:n}x\right)
}{\overline{F}\left(  Z_{n-k:n}\right)  },x^{-1/\gamma_{1}}\right\}  .
\]
The uniform convergence version corresponding the second-order condition of
regularly varying functions $\left(  \ref{second-orderF}\right)  $ says that,
for each $\epsilon>0,$ there exists $t_{0}$ such that for $t>t_{0},$ we have%
\[
\frac{\overline{F}\left(  tx\right)  }{\overline{F}\left(  t\right)
}-x^{-1/\gamma_{1}}=\left(  1+o\left(  x^{-\epsilon}\right)  \right)
A_{1}\left(  t\right)  x^{-1/\gamma_{1}}\frac{x^{\tau_{1}/\gamma_{1}}-1}%
{\tau_{1}\gamma_{1}},
\]
for any $x>1;$ see, e.g., Proposition 4 together with Remark 1 in
\cite{HJ2011}. Then letting $t=Z_{n-k:n}$ yields that
\[
\frac{\overline{F}\left(  Z_{n-k:n}x\right)  }{\overline{F}\left(
Z_{n-k:n}\right)  }-x^{-1/\gamma_{1}}=\left(  1+o_{\mathbf{P}}\left(
x^{-\epsilon}\right)  \right)  A_{1}\left(  Z_{n-k:n}\right)  x^{-1/\gamma
_{1}}\frac{x^{\tau_{1}/\gamma_{1}}-1}{\tau_{1}\gamma_{1}},\text{ as
}n\rightarrow\infty.
\]
Recall that $Z_{n-k:n}/h\overset{\mathbf{P}}{\rightarrow}1,$ then the regular
variation of $A_{1}$ implies that
\[
\frac{\overline{F}\left(  Z_{n-k:n}x\right)  }{\overline{F}\left(
Z_{n-k:n}\right)  }-x^{-1/\gamma_{1}}=\left(  1+o_{\mathbf{P}}\left(
x^{-\epsilon}\right)  \right)  A_{1}\left(  h\right)  x^{-1/\gamma_{1}}%
\frac{x^{\tau_{1}/\gamma_{1}}-1}{\tau_{1}\gamma_{1}}.
\]
It follows that%
\[
\sqrt{k}T_{n1}^{\left(  1\right)  }=\left(  \beta/p\right)  ^{2}\sqrt{k}%
A_{1}\left(  h\right)  \int_{1}^{\infty}\left(  1+o_{\mathbf{P}}\left(
x^{-\epsilon}\right)  \right)  x^{-\beta/\gamma_{1}-1}\frac{x^{\tau_{1}%
/\gamma_{1}}-1}{\tau_{1}\gamma_{1}}dx.
\]
By assumption, we have $\sqrt{k}A_{1}\left(  h\right)  =O\left(  1\right)  ,$
then it is easy to verify that the latter integral equals $\left(
1+o_{\mathbf{P}}\left(  1\right)  \right)  /\left(  \left(  \beta/p\right)
\left(  \beta/p-\tau_{1}\right)  \right)  ,$ therefore%
\[
\sqrt{k}T_{n1}^{\left(  1\right)  }=\frac{\beta}{\beta-p\tau_{1}}\sqrt{k}%
A_{1}\left(  h\right)  +o_{\mathbf{P}}\left(  1\right)  .
\]
Now, we consider the remainder term $T_{n1}^{\left(  2\right)  }.$ Making use
of statement $\left(  \ref{second-orderF}\right)  $ we show that $\rho
_{n}\left(  x\right)  =x^{-p/\gamma}\left(  1+o_{\mathbf{P}}\left(
x^{-\epsilon}\right)  \right)  ,$ as $n\rightarrow\infty,$ uniformly over
$x>1.$ It follows that
\begin{align*}
\sqrt{k}T_{n1}^{\left(  1\right)  } &  =\frac{\left(  \beta/p-1\right)
\left(  \beta/p\right)  ^{2}}{2}\\
&  \sqrt{k}\int_{1}^{\infty}x^{-1}\left(  \left(  1+o_{\mathbf{P}}\left(
x^{-\epsilon}\right)  \right)  A_{1}\left(  Z_{n-k:n}\right)  x^{-1/\gamma
_{1}}\frac{x^{\tau_{1}/\gamma_{1}}-1}{\tau_{1}\gamma_{1}}\right)  ^{2}\\
&  \ \ \ \ \ \ \ \ \ \ \ \ \ \ \ \ \times\left(  x^{-p/\gamma}\left(
1+o_{\mathbf{P}}\left(  x^{-\epsilon}\right)  \right)  \right)  ^{\beta
/p-2}dx.
\end{align*}
Elementary calculation gives $\sqrt{k}T_{n1}^{\left(  1\right)  }%
=O_{\mathbf{P}}\left(  1/\sqrt{k}\right)  \left(  \sqrt{k}A_{1}\left(
Z_{n-k:n}\right)  \right)  ^{2}$ which tends to zero in probability as
$n\rightarrow\infty,$ because $1/\sqrt{k}\rightarrow0$ and $\sqrt{k}%
A_{1}\left(  Z_{n-k:n}\right)  =O_{\mathbf{P}}\left(  1\right)  .$ Let us now
consider the second and third terms $T_{n2}$\textbf{ }and\textbf{ }$T_{n3}.$
We mentioned in $\left(  \ref{Tn1-Tn2}\right)  $ that $T_{n2}\mathbf{+}%
T_{n3}=o_{\mathbf{P}}\left(  \widehat{p}_{k}-p\right)  ,$ it follows that
$\sqrt{k}\left(  T_{n2}\mathbf{+}T_{n3}\right)  =o_{\mathbf{P}}\left(
1\right)  \sqrt{k}\left(  \widehat{p}_{k}-p\right)  .$ The main term $T_{n4}$
leads to the asymptotic normality of our estimator. Indeed, making Taylor's
expansion to the second order, yields that%
\begin{align*}
&  \left(  \frac{\overline{F}_{n}^{\left(  NA\right)  }\left(  Z_{n-k:n}%
x\right)  }{\overline{F}_{n}^{\left(  NA\right)  }\left(  Z_{n-k:n}\right)
}\right)  ^{\beta/p}-\left(  \frac{\overline{F}\left(  Z_{n-k:n}x\right)
}{\overline{F}\left(  Z_{n-k:n}\right)  }\right)  ^{\beta/p}\\
&  =\frac{\beta}{p}\left(  \frac{\overline{F}\left(  Z_{n-k:n}x\right)
}{\overline{F}\left(  Z_{n-k:n}\right)  }\right)  ^{\beta/p-1}L_{n}\left(
x\right)  +\frac{\left(  \beta/p\right)  \left(  \beta/p-1\right)  }{2}\left(
\eta_{n}\left(  x\right)  \right)  ^{\beta/p-2}L_{n}^{2}\left(  x\right)  .
\end{align*}
This allows to split $T_{n4}$ into the sum of%
\[
T_{n4}^{\left(  1\right)  }:=\left(  \beta/p\right)  ^{2}\int_{1}^{\infty
}x^{-1}\left(  \frac{\overline{F}\left(  Z_{n-k:n}x\right)  }{\overline
{F}\left(  Z_{n-k:n}\right)  }\right)  ^{\beta/p-1}L_{n}\left(  x\right)  dx,
\]
and%
\[
T_{n4}^{\left(  2\right)  }:=\frac{\left(  \beta/p\right)  \left(
\beta/p-1\right)  }{2}\int_{1}^{\infty}x^{-1}\left(  \eta_{n}\left(  x\right)
\right)  ^{\beta/p-2}L_{n}^{2}\left(  x\right)  dx.
\]
\cite{MNS2025} state in that there exists a sequence of standard Wiener
processes $\left\{  W_{n}\left(  s\right)  ;\text{ }0\leq s\leq1\right\}
_{n\geq1}$ defined on the probability space $\left(  \Omega,\mathcal{A}%
,\mathbf{P}\right)  ,$ such that for every small $\epsilon>0$ and for any
$1/4<\eta<1/2:$%
\begin{equation}
\sqrt{k}L_{n}=J_{n}\left(  x\right)  +o_{\mathbf{P}}\left(  x^{\left(
2\eta-p\right)  /\gamma-\epsilon}\right)  ,\label{Gauss}%
\end{equation}
uniformly over $x>1,$ where $J_{n}\left(  x\right)  =J_{1n}\left(  x\right)
+J_{2n}\left(  x\right)  ,$ with%
\begin{equation}
J_{1n}\left(  x\right)  :=x^{1/\gamma_{2}}\mathbf{W}_{n,1}\left(
x^{-1/\gamma}\right)  -x^{-1/\gamma_{1}}\mathbf{W}_{n,1}\left(  1\right)
,\label{Jn1}%
\end{equation}
and
\begin{equation}
J_{2n}\left(  x\right)  :=\gamma^{-1}x^{-1/\gamma_{1}}%
{\displaystyle\int_{1}^{x}}
u^{1/\gamma-1}\left(  p\mathbf{W}_{n,2}\left(  u^{-1/\gamma}\right)
-q\mathbf{W}_{n,1}\left(  u^{-1/\gamma}\right)  \right)  du,\label{Jn2}%
\end{equation}
where%
\[
\mathbf{W}_{n,1}\left(  s\right)  :=W_{n}\left(  \theta\right)  -W_{n}\left(
\theta-ps\right)  ,\text{ for }0<s\leq\theta/p,
\]
and%
\[
\mathbf{W}_{n,2}\left(  s\right)  :=W_{n}\left(  1\right)  -W_{n}\left(
1-qs\right)  ,\text{ for }0\leq s\leq1,
\]
being two independent centred Gaussian processes, with$\ \theta:=H^{\left(
1\right)  }\left(  \infty\right)  $ and $q:=1-p.$ To obtain the weak
approximation $\left(  \ref{Gauss}\right)  ,$ the authors used the
representations of $H_{n}^{\left(  0\right)  }$ and $H_{n}^{\left(  1\right)
}$ by a common uniform empirical process \citep[] []{Deheuvels1996} pertaining
to the sample of uniform-$\left(  0,1\right)  $ rv's%
\begin{equation}
U_{i}:=\delta_{i}H^{\left(  1\right)  }\left(  Z_{i}\right)  +\left(
1-\delta_{i}\right)  \left(  \theta+H^{\left(  0\right)  }\left(
Z_{i}\right)  \right)  ,\text{ }i=1,...,n.\label{unif}%
\end{equation}
Then, they applied the approximations of $H_{n}^{\left(  0\right)  }$ and
$H_{n}^{\left(  1\right)  }$ by the common sequence of Brownian bridges
$W_{n}\left(  s\right)  -sW_{n}\left(  1\right)  $ \citep[] []{CSCSHM86}. Let
us now turn to our problem. We simultaneously use $\left(  \ref{Gauss}\right)
$ and $\left(  \ref{appF}\right)  ,$ to write
\begin{align*}
\sqrt{k}T_{n4}^{\left(  1\right)  } &  =\left(  \beta/p\right)  ^{2}\int%
_{1}^{\infty}x^{-1}\left(  x^{-p/\gamma}\left(  1+o_{\mathbf{P}}\left(
x^{-\epsilon}\right)  \right)  \right)  ^{\beta/p-1}\left(  J_{n}\left(
x\right)  +o_{\mathbf{P}}\left(  x^{\left(  2\eta-p\right)  /\gamma-\epsilon
}\right)  \right)  dx\\
&  =\left(  \beta/p\right)  ^{2}\int_{1}^{\infty}x^{-1-\left(  \beta-p\right)
/\gamma}\left(  1+o_{\mathbf{P}}\left(  x^{-\epsilon}\right)  \right)  \left(
J_{n}\left(  x\right)  +o_{\mathbf{P}}\left(  x^{\left(  2\eta-p\right)
/\gamma-\epsilon}\right)  \right)  dx,
\end{align*}
which may be rewritten into the sum of $\left(  \beta/p\right)  ^{2}\int%
_{1}^{\infty}x^{-1-\left(  \beta-p\right)  /\gamma}J_{n}\left(  x\right)
dx,$
\[
R_{1,n,\epsilon}:=o_{\mathbf{P}}\left(  1\right)  \int_{1}^{\infty
}x^{-1-\left(  \beta-p\right)  /\gamma-\epsilon}J_{n}\left(  x\right)
dx\text{ and }R_{2,\epsilon,\eta}:=o_{\mathbf{P}}\left(  1\right)  \int%
_{1}^{\infty}x^{-1-\left(  \beta-2\eta\right)  /\gamma-2\epsilon}dx.
\]
From assertion $\left(  i\right)  $ of Proposition \ref{prop3}$,$ we have
$J_{n}\left(  x\right)  =O_{\mathbf{P}}\left(  x^{\left(  2\eta-p\right)
/\gamma}\right)  ,$ as $n\rightarrow\infty,$ for every $1/4<\eta<1/2$ and
uniformly over $x>1.$ It follows that
\[
R_{1,n,\epsilon}:=o_{\mathbf{P}}\left(  1\right)  \int_{1}^{\infty
}x^{-1+\left(  2\eta-\beta\right)  /\gamma-2\epsilon}dx,
\]
which equals $o_{\mathbf{P}}\left(  1\right)  $ because $\beta>1$ and
$1/4<\eta<1/2.$ Using similar arguments leads to $R_{2,\epsilon,\eta
}=o_{\mathbf{P}}\left(  1\right)  $ as well. Next, we show that $\sqrt
{k}T_{n4}^{\left(  2\right)  }=o_{\mathbf{P}}\left(  1\right)  $ too. To this
end, we write%
\[
\sqrt{k}T_{n4}^{\left(  2\right)  }=k^{-1/2}\frac{\left(  \beta/p\right)
\left(  \beta/p-1\right)  }{2}\int_{1}^{\infty}x^{-1}\left(  \eta_{n}\left(
x\right)  \right)  ^{\beta/p-2}\left(  \sqrt{k}L_{n}\left(  x\right)  \right)
^{2}dx.
\]
Once again, making use of $\left(  \ref{appeta}\right)  $ and assertion
$\left(  ii\right)  $ of Proposition \ref{prop3} together, yields that%
\begin{align*}
\sqrt{k}T_{n4}^{\left(  2\right)  } &  =O_{\mathbf{P}}\left(  k^{-1/2}\right)
\int_{1}^{\infty}x^{-1}\left(  x^{-p/\gamma}\left(  1+o_{\mathbf{P}}\left(
x^{2\eta/\gamma-\epsilon}\right)  \right)  \right)  ^{\beta/p-2}\left(
x^{\left(  2\eta-p\right)  /\gamma-\epsilon}\right)  ^{2}dx\\
&  =o_{\mathbf{P}}\left(  1\right)  \int_{1}^{\infty}x^{-1-\left(  \beta
-4\eta\right)  /\gamma-2\epsilon}\left(  \left(  1+o_{\mathbf{P}}\left(
x^{2\eta/\gamma-\epsilon}\right)  \right)  \right)  dx.
\end{align*}
The latter integral is finite provided that $\beta>4\eta$ which holds because
$\eta>1/4$ and $\beta>1.$ Finally, we showed that for every $\beta>1,$ we have%
\begin{equation}
\sqrt{k}\left(  \widehat{\gamma}_{1}^{\left(  NA\right)  }\left(
\beta\right)  -\gamma_{1}\right)  =N_{n}\left(  \beta\right)  +\frac{\beta
}{\beta-p\tau_{1}}\sqrt{k}A_{1}\left(  h\right)  +o_{\mathbf{P}}\left(
\sqrt{k}\left(  \widehat{p}_{k}-p\right)  \right)  +o_{\mathbf{P}}\left(
1\right)  ,\label{final-app}%
\end{equation}
where $N_{n}\left(  \beta\right)  :=\left(  \beta/p\right)  ^{2}\int%
_{1}^{\infty}x^{-1-\left(  \beta-p\right)  /\gamma}J_{n}\left(  x\right)  dx.$
Note that $\sqrt{k}\left(  \widehat{p}_{k}-p\right)  $ is an asymptotically
Gaussian rv (see, e.g., \cite{EnFG08} or Theorem 2.1 in \cite{BMN-2015}), it
follows that%
\[
\sqrt{k}\left(  \widehat{\gamma}_{1,k}^{\left(  NA\right)  }\left(
\beta\right)  -\gamma_{1}\right)  =N_{n}\left(  \beta\right)  +\frac{\beta
}{\beta-p\tau_{1}}\sqrt{k}A_{1}\left(  h\right)  +o_{\mathbf{P}}\left(
1\right)  .
\]
We show in Proposition \ref{prop4} that
\[
N_{n}\left(  \beta\right)  =\frac{\beta\left(  p+\beta-1\right)  }{p}\int%
_{0}^{1}s^{\beta-2}\mathbf{W}_{n,1}\left(  s\right)  ds+\frac{\beta}%
{p}\mathbf{W}_{n,1}\left(  1\right)  +\beta\int_{0}^{1}s^{\beta-2}%
\mathbf{W}_{n,2}\left(  s\right)  ds,
\]
which meets assertion $\left(  \ref{Gaussian-app}\right)  $ of Theorem
\ref{theorem 2}. Let us make sure that $N_{n}\left(  \beta\right)  $ is
stochastically bounded. From assertion $\left(  \ref{sup}\right)  $ we have
$\sup_{0<s\leq1}s^{2\eta-1}\left\vert \mathbf{W}_{n,i}\left(  s\right)
\right\vert =O_{\mathbf{P}}\left(  1\right)  ,$ $i=1,2,$ for every
$1/4<\eta<1/2,$ it follows that
\[
\int_{0}^{1}s^{\beta-2}\mathbf{W}_{n,i}\left(  s\right)  ds=\int_{0}%
^{1}s^{\beta-1-2\eta}\left\{  s^{2\eta-1}\mathbf{W}_{n,i}\left(  s\right)
\right\}  ds=O_{\mathbf{P}}\left(  1\right)  \int_{0}^{1}s^{\beta-1-2\eta}ds.
\]
Since $2\eta>1/2$ and $\beta>1,$ then $\beta-2\eta>0,$ it leads to $\int%
_{0}^{1}s^{\beta-2}\mathbf{W}_{n,i}\left(  s\right)  ds=O_{\mathbf{P}}\left(
1\right)  ,$ $i=1,2,$ therefore $N_{n}\left(  \beta\right)  =O_{\mathbf{P}%
}\left(  1\right)  .$ Moreover, being a linear functional of two (independent)
sequences of centred Gaussian processes $\mathbf{W}_{n,1}$ and $\mathbf{W}%
_{n,2},$ $N_{n}\left(  \beta\right)  $ is a sequence of centred Gaussian rv's
as well. We state in Proposition \ref{prop4} that the variance of
$N_{n}\left(  \beta\right)  $ equals $\sigma_{\beta}^{2}:=\frac{1}{p}%
\frac{\beta^{2}}{2\beta-1}\gamma_{1}^{2},$ which completes the
proof.\textbf{\newpage}

\noindent\textbf{Conclusion\smallskip}

\noindent Here are some features of our new estimator $\widehat{\gamma}%
_{1,k}^{\left(  \bullet\right)  }\left(  \beta\right)  :$

\begin{itemize}
\item It is valid for any censoring proportion $0<p<1$ provided that the
tuning parameter is chosen as $\beta=1+\epsilon,$ where $\epsilon>0$ is
sufficiently small (i.e., $\epsilon=10^{-2}$ or $\epsilon=10^{-3}).$

\item Simulation studies demonstrate that it outperforms existing estimators
in terms of both bias and MSE under weak and strong censoring regimes (i.e.,
$p>1/2$ and $p\leq1/2$ respectively).

\item Its asymptotic variance is smaller than those of $\widehat{\gamma}%
_{1,k}^{\left(  MNS\right)  }$ \citep[] []{MNS2025} and $\gamma_{1,k}^{\left(
W\right)  }$ \citep[] []{WW2014}.

\item Although its asymptotic variance is larger than that of $\widehat{\gamma
}_{1,k}^{\left(  EFG\right)  }$\citep[] []{EnFG08} (equal to $p^{-1}\gamma
_{1}^{2}),$ its MSE is consistently smaller.

\item Previous results on tail index estimation for right-censored Pareto-type
data, derived using Kaplan-Meier estimator, were established under Hall-type
model. Thanks to Proposition \ref{prop7}$,$ Theorems \ref{theorem 1} and
\ref{theorem 2}$,$ things can now be extended to the broader class of
regularly varying functions.

\item The slight increase, which might occur in its asymptotic bias, can be
mitigated through bias reduction techniques of
\citep[see, e.g.,] []{Caeiro2005}. More precisely, a bias-corrected and
consistent estimator to $\gamma_{1},$ denoted $\overline{\gamma}_{1,k}\left(
\beta\right)  ,$ may be constructed based on $\widehat{\gamma}_{1,k}^{\left(
\bullet\right)  }\left(  \beta\right)  ,$ such that
\[
\sqrt{k}\left(  \overline{\gamma}_{1,k}\left(  \beta\right)  -\gamma
_{1}\right)  \overset{\mathcal{D}}{\rightarrow}\mathcal{N}\left(
0,p^{-1}\beta^{2}\left(  2\beta-1\right)  \gamma_{1}^{2}\right)  ,\text{ as
}n\rightarrow\infty.
\]
This refinement will be explored in future work.\textbf{\medskip}
\end{itemize}

\noindent\textbf{Declarations of competing interest\smallskip}

\noindent The authors declare that they have no known competing financial
interests or personal relationships that could influence the work reported in
this paper.

\section{\textbf{Appendix A \label{Appendix A}}}

\begin{proposition}
\textbf{\label{prop1}}Assume that cdf $F$ satisfies the first-order condition
of regular variation $\left(  \ref{RV}\right)  .$ Then, for every small
$0<\epsilon<1$ and all large $t,$ we have
\[
\frac{\overline{F}\left(  ty\right)  }{\overline{F}\left(  t\right)
}=y^{-1/\gamma_{1}}\left(  1+o\left(  y^{-\epsilon}\right)  \right)  ,
\]
uniformly over $y>1.$
\end{proposition}

\begin{proof}
This is straightforward from the application of Potter's inequalities (see
e.g. Proposition B.1.10 in \cite{deHF06}) to the regularly varying function
$\overline{F}.$
\end{proof}

\begin{proposition}
\textbf{\label{prop2}}Under the assumptions of Proposition \ref{prop5}, we
have, for every $\beta>0,$%
\[
\left(  i\right)  \text{ }\left(  \beta/p\right)  \int_{1}^{\infty}%
y^{-1}\left(  \frac{\overline{F}\left(  ty\right)  }{\overline{F}\left(
t\right)  }\right)  ^{\beta/p}dy\rightarrow\gamma_{1},\text{ as }%
t\rightarrow\infty,
\]
which is equivalent to%
\[
\left(  ii\right)  \text{ }\left(  \beta/p\right)  ^{2}\int_{1}^{\infty
}\left(  \frac{\overline{F}\left(  ty\right)  }{\overline{F}\left(  y\right)
}\right)  ^{\beta/p-1}\left(  \log x\right)  d\frac{F\left(  ty\right)
}{\overline{F}\left(  t\right)  }\rightarrow\gamma_{1},\text{ as }%
t\rightarrow\infty.
\]

\end{proposition}

\begin{proof}
In view of Proposition \ref{prop1}, for small $\epsilon>0$ and large $t,$ we
have%
\[
\left(  \beta/p\right)  \int_{1}^{\infty}y^{-1}\left(  \frac{\overline
{F}\left(  ty\right)  }{\overline{F}\left(  t\right)  }\right)  ^{\beta
/p}dy=\left(  \beta/p\right)  \int_{1}^{\infty}y^{-1}\left(  y^{-1/\gamma_{1}%
}\left(  1+o\left(  y^{-\epsilon}\right)  \right)  \right)  ^{\beta/p}dy,
\]
which is equal to $\left(  \beta/p\right)  \int_{1}^{\infty}y^{-1-\beta
/\left(  p\gamma_{1}\right)  }dy+o\left(  1\right)  \int_{1}^{\infty
}y^{-1-\beta/\left(  p\gamma_{1}\right)  -\epsilon}dy.$ It is clear that that
the first term equals $\gamma_{1}$ and the second one tends to zero as
$t\rightarrow\infty.$ To show assertion $\left(  ii\right)  ,$ its suffices to
make an integration by parts to show that%
\[
\left(  \beta/p\right)  \int_{1}^{\infty}y^{-1}\left(  \frac{\overline
{F}\left(  ty\right)  }{\overline{F}\left(  y\right)  }\right)  ^{\beta
/p}dy=\left(  \beta/p\right)  ^{2}\int_{1}^{\infty}\left(  \frac{\overline
{F}\left(  ty\right)  }{\overline{F}\left(  y\right)  }\right)  ^{\beta
/p}\left(  \log y\right)  d\frac{F\left(  ty\right)  }{\overline{F}\left(
t\right)  },
\]
then use assertion $\left(  i\right)  $.
\end{proof}

\begin{proposition}
\textbf{\label{prop3}}For every $1/4<\eta<1/2$ and small $0<\epsilon<1,$ we
have%
\[
\left(  i\right)  \text{ }J_{n}\left(  x\right)  =O_{\mathbf{P}}\left(
x^{\left(  2\eta-p\right)  /\gamma}\right)  ,\text{ }\left(  ii\right)  \text{
}\sqrt{k}L_{n}\left(  x\right)  =O_{\mathbf{P}}\left(  x^{\left(
2\eta-p\right)  /\gamma-\epsilon}\right)  ,
\]
and%
\[
\left(  iii\right)  \text{ }\frac{\overline{F}_{n}^{\left(  NA\right)
}\left(  Z_{n-k:n}x\right)  }{\overline{F}_{n}^{\left(  NA\right)  }\left(
Z_{n-k:n}\right)  }=x^{-p/\gamma}\left(  1+o_{\mathbf{P}}\left(
x^{2\eta/\gamma-\epsilon}\right)  \right)  ,\text{ as }n\rightarrow\infty,
\]
uniformly over $x>1.$
\end{proposition}

\begin{proof}
From Lemma 3.2 in \cite{EHL2006}, we have $\sup_{0<s\leq1}s^{\nu}\left\vert
W_{n}\left(  s\right)  \right\vert =O_{\mathbf{P}}\left(  1\right)  ,$ for
every $0<\nu<1/2.$ In other words, we have
\[
\sup_{0<s\leq1}\frac{\left\vert W_{n}\left(  s\right)  \right\vert
}{s^{1-2\eta}}=O_{\mathbf{P}}\left(  1\right)  ,\text{ as }n\rightarrow\infty,
\]
for every $1/4<\eta<1/2.$ On the other hand, for each $n,$ we jointly have%
\[
\mathbf{W}_{n,1}\left(  s\right)  \overset{\mathcal{D}}{=}W_{n}\left(
ps\right)  \overset{\mathcal{D}}{=}p^{1/2}W_{n}\left(  s\right)  \text{ and
}\mathbf{W}_{n,2}\left(  s\right)  \overset{\mathcal{D}}{=}W_{n}\left(
qs\right)  \overset{\mathcal{D}}{=}q^{1/2}W_{n}\left(  s\right)  ,\text{
}0<s\leq1,
\]
with $\overset{\mathcal{D}}{=}$ standing for equality in distribution. It
follows that, for $i=1,2,$ we have
\begin{equation}
\sup_{0<s\leq1}\frac{\left\vert \mathbf{W}_{n,i}\left(  s\right)  \right\vert
}{s^{1-2\eta}}=O_{\mathbf{P}}\left(  1\right)  ,\text{ as }n\rightarrow
\infty.\label{sup}%
\end{equation}
Recall that $\gamma_{2}=\left(  1-p\right)  /\gamma$ and $\gamma_{1}%
=p/\gamma,$ then%
\[
J_{1n}\left(  x\right)  =x^{\left(  1-p\right)  /\gamma}\mathbf{W}%
_{n,1}\left(  x^{-1/\gamma}\right)  -x^{-p/\gamma}\mathbf{W}_{n,1}\left(
1\right)  .
\]
By applying assertion $\left(  \ref{sup}\right)  ,$ for $i=1,$ we readily
verify that $J_{1n}\left(  x\right)  =O_{\mathbf{P}}\left(  x^{\left(
2\eta-p\right)  /\gamma}\right)  ,$ uniformly over $x>1.$ Likewise, we write%
\[
J_{2n}\left(  x\right)  =\dfrac{x^{-p/\gamma}}{\gamma}%
{\displaystyle\int_{1}^{x}}
u^{1/\gamma-1}\left\{  p\mathbf{W}_{n,2}\left(  u^{-1/\gamma}\right)
+q\mathbf{W}_{n,1}\left(  u^{-1/\gamma}\right)  \right\}  du.
\]
On again, by using assertion $\left(  \ref{sup}\right)  $ with $i=2,$ we also
show that $J_{2n}\left(  x\right)  =O_{\mathbf{P}}\left(  x^{\left(
2\eta-p\right)  /\gamma}\right)  ,$ uniformly over $x>1.$ In summary, we
showed that for ever $1/4<\eta<1/2,$ we have $J_{n}\left(  x\right)
=O_{\mathbf{P}}\left(  x^{\left(  2\eta-p\right)  /\gamma}\right)  ,$
uniformly over $x>1,$ which completes the proof of assertion $\left(
i\right)  .$ Now, we consider the second assertion. From \ the Gaussian
approximation $\left(  \ref{Gauss}\right)  ,$ we have
\[
\sqrt{k}L_{n}=J_{n}\left(  x\right)  +o_{\mathbf{P}}\left(  x^{\left(
2\eta-p\right)  /\gamma-\epsilon}\right)  ,
\]
then from assertion $\left(  i\right)  ,$ we infer that $\sqrt{k}%
L_{n}=O_{\mathbf{P}}\left(  x^{\left(  2\eta-p\right)  /\gamma-\epsilon
}\right)  .$ Assertion $\left(  iii\right)  $ is a consequence of the first
and second one. Indeed, by $\left(  \ref{Ln}\right)  ,$ we can write
\[
\frac{\overline{F}_{n}^{\left(  NA\right)  }\left(  Z_{n-k:n}x\right)
}{\overline{F}_{n}^{\left(  NA\right)  }\left(  Z_{n-k:n}\right)  }%
=L_{n}\left(  x\right)  -\frac{\overline{F}\left(  Z_{n-k:n}x\right)
}{\overline{F}\left(  Z_{n-k:n}\right)  }.
\]
Note that assertion $\left(  ii\right)  $\ implies that $\sqrt{k}%
L_{n}=o_{\mathbf{P}}\left(  x^{\left(  2\eta-p\right)  /\gamma-\epsilon
}\right)  $ because $1/\sqrt{k}\rightarrow0,$ and from Proposition
\ref{prop1}, we deduce that
\[
\frac{\overline{F}\left(  Z_{n-k:n}x\right)  }{\overline{F}\left(
Z_{n-k:n}\right)  }=x^{-1/\gamma_{1}}\left(  1+o_{\mathbf{P}}\left(
x^{-\epsilon}\right)  \right)  ,
\]
then combining the two results yields assertion $\left(  iii\right)  ,$ as sought.
\end{proof}

\begin{proposition}
\textbf{\label{prop4}}For every $\beta>1$ and $0<p<1,$ we have%
\[
N_{n}\left(  \beta\right)  =\frac{\beta\left(  p+\beta-1\right)  }{p}\int%
_{0}^{1}s^{\beta-2}\mathbf{W}_{n,1}\left(  s\right)  ds+\beta\int_{0}%
^{1}s^{\beta-2}\mathbf{W}_{n,2}\left(  s\right)  ds+\frac{\beta}{p}%
\mathbf{W}_{n,1}\left(  1\right)  .
\]
Moreover, we have $\mathbf{Var}\left[  N_{n}\left(  \beta\right)  \right]
=p^{-1}\beta^{2}/\left(  2\beta-1\right)  .$
\end{proposition}

\begin{proof}
Recall that $N_{n}\left(  \beta\right)  :=\left(  \beta/p\right)  ^{2}\int%
_{1}^{\infty}x^{-1-\left(  \beta-p\right)  /\gamma}J_{n}\left(  x\right)  dx,$
with $J_{n}\left(  x\right)  =J_{1n}\left(  x\right)  +J_{2n}\left(  x\right)
,$ where%
\[
J_{1n}\left(  x\right)  :=x^{1/\gamma_{2}}\mathbf{W}_{n,1}\left(
x^{-1/\gamma}\right)  -x^{-1/\gamma_{1}}\mathbf{W}_{n,1}\left(  1\right)  ,
\]
and
\[
J_{2n}\left(  x\right)  :=\dfrac{x^{-1/\gamma_{1}}}{\gamma}%
{\displaystyle\int_{1}^{x}}
u^{1/\gamma-1}\left\{  p\mathbf{W}_{n,2}\left(  u^{-1/\gamma}\right)
-q\mathbf{W}_{n,1}\left(  u^{-1/\gamma}\right)  \right\}  du.
\]
The change of variables $x=t^{-\gamma_{1}}\left(  =t^{-\gamma/p}\right)  $
yields that%
\[
N_{n}\left(  \beta\right)  =\left(  \beta/p\right)  ^{2}\gamma_{1}\int_{0}%
^{1}t^{\beta/p-2}J_{n}\left(  t^{-\gamma/p}\right)  dt,
\]
where $J_{n}\left(  t^{-\gamma/p}\right)  =J_{n1}\left(  t^{-\gamma/p}\right)
+J_{2n}\left(  t^{-\gamma/p}\right)  ,$ with
\[
J_{n1}\left(  t^{-\gamma/p}\right)  =t^{-q/p}\mathbf{W}_{n,1}\left(
t^{1/p}\right)  -t\mathbf{W}_{n,1}\left(  1\right)  ,
\]
and
\[
J_{2n}\left(  t^{-\gamma/p}\right)  =\dfrac{t}{\gamma}%
{\displaystyle\int_{1}^{t^{-\gamma/p}}}
u^{1/\gamma-1}\left\{  p\mathbf{W}_{n,2}\left(  u^{-1/\gamma}\right)
-q\mathbf{W}_{n,1}\left(  u^{-1/\gamma}\right)  \right\}  du.
\]
Thus, $N_{n}\left(  \beta\right)  $ is the sum of
\[
N_{n1}\left(  \beta\right)  :=\left(  \beta/p\right)  ^{2}\gamma_{1}\int%
_{0}^{1}t^{\beta/p-2}\left\{  t^{-q/p}\mathbf{W}_{n,1}\left(  t^{1/p}\right)
-t\mathbf{W}_{n,1}\left(  1\right)  \right\}  dt,
\]
and%
\[
N_{n2}\left(  \beta\right)  :=\frac{\left(  \beta/p\right)  ^{2}\gamma_{1}%
}{\gamma}\int_{0}^{1}t^{\beta/p-1}\left\{
{\displaystyle\int_{1}^{t^{-\gamma/p}}}
u^{1/\gamma-1}\left\{  p\mathbf{W}_{n,2}\left(  u^{-1/\gamma}\right)
-q\mathbf{W}_{n,1}\left(  u^{-1/\gamma}\right)  \right\}  du\right\}  dt.
\]

Let us write
\[
N_{n2}\left(  \beta\right)  =\frac{\left(  \beta/p\right)  \gamma_{1}}{\gamma
}\int_{0}^{1}\left\{
{\displaystyle\int_{1}^{t^{-\gamma/p}}}
u^{1/\gamma-1}\left\{  p\mathbf{W}_{n,2}\left(  u^{-1/\gamma}\right)
-q\mathbf{W}_{n,1}\left(  u^{-1/\gamma}\right)  \right\}  du\right\}
dt^{\beta/p}.
\]
We, by an integration by parts, that%
\[
N_{n2}\left(  \beta\right)  =\frac{\beta}{p^{2}}\gamma_{1}\int_{0}%
^{1}t^{\left(  \beta-1\right)  /p-1}\left\{  p\mathbf{W}_{n,2}\left(
t^{1/p}\right)  -q\mathbf{W}_{n,1}\left(  t^{1/p}\right)  \right\}  dt.
\]
Once again, the change of variables $s=t^{1/p}$ yields that%
\[
N_{n1}\left(  \beta\right)  =p\left(  \beta/p\right)  ^{2}\gamma_{1}\int%
_{0}^{1}s^{\beta-2}\left\{  \mathbf{W}_{n,1}\left(  s\right)  -s\mathbf{W}%
_{n,1}\left(  1\right)  \right\}  ds,
\]
and%
\[
N_{n2}\left(  \beta\right)  =\frac{\beta}{p}\gamma_{1}\int_{0}^{1}s^{\beta
-2}\left\{  p\mathbf{W}_{n,2}\left(  s\right)  -q\mathbf{W}_{n,1}\left(
s\right)  \right\}  ds.
\]
It is easy to verify that summing $N_{n1}\left(  \beta\right)  $ and
$N_{n2}\left(  \beta\right)  $ yields the first result of Proposition
\ref{prop4}. Let us now compute the variance of the centred rv $N_{n}\left(
\beta\right)  .$ We have%
\begin{align*}
\gamma_{1}^{-2}\mathbf{E}\left[  N_{n}\left(  \beta\right)  \right]  ^{2}  &
=\left(  \frac{\beta\left(  p+\beta-1\right)  }{p}\right)  ^{2}\mathbf{E}%
\left[  \int_{0}^{1}s^{\beta-2}\mathbf{W}_{n,1}\left(  s\right)  ds\right]
^{2}\\
&  +\beta^{2}\mathbf{E}\left[  \int_{0}^{1}s^{\beta-2}\mathbf{W}_{n,2}\left(
s\right)  ds\right]  ^{2}+\left(  \frac{\beta}{p}\right)  ^{2}\mathbf{E}%
\left[  \mathbf{W}_{n,1}\left(  1\right)  \right]  ^{2}\\
&  \ \ \ \ \ \ \ \ \ \ \ \ \ \ \ \ \ \ +2\frac{\beta\left(  p+\beta-1\right)
}{p}\frac{\beta}{p}\int_{0}^{1}s^{\beta-2}\mathbf{E}\left[  \mathbf{W}%
_{n,1}\left(  s\right)  \mathbf{W}_{n,1}\left(  1\right)  \right]  ds.
\end{align*}
It is easy to check that%
\[
p^{-1}\mathbf{E}\left[  \mathbf{W}_{n,1}\left(  s\right)  \mathbf{W}%
_{n,1}\left(  t\right)  \right]  =\min\left(  s,t\right)  =q^{-1}%
\mathbf{E}\left[  \mathbf{W}_{n,2}\left(  s\right)  \mathbf{W}_{n,2}\left(
t\right)  \right]  ,
\]
and $\mathbf{E}\left[  \mathbf{W}_{n,1}\left(  s\right)  \mathbf{W}%
_{n,2}\left(  t\right)  \right]  =0.$ By using the covariances above, we show
that%
\[
p^{-1}\mathbf{E}\left(  \int_{0}^{1}s^{\beta-2}\mathbf{W}_{n,1}\left(
s\right)  ds\right)  ^{2}=\frac{2}{\beta\left(  2\beta-1\right)  }%
=q^{-1}\mathbf{E}\left(  \int_{0}^{1}s^{\beta-2}\mathbf{W}_{n,2}\left(
s\right)  ds\right)  ^{2},
\]%
\[
2\int_{0}^{1}s^{\beta-2}\mathbf{E}\left[  \mathbf{W}_{n,1}\left(  s\right)
\mathbf{W}_{n,1}\left(  1\right)  \right]  ds=\frac{2p}{\beta}\text{ and
}\mathbf{E}\left[  \mathbf{W}_{n,1}\left(  1\right)  \right]  ^{2}=p.
\]
Finally, we end up with%
\begin{align*}
\mathbf{E}\left[  N_{n}\left(  \beta\right)  \right]  ^{2}  &  =\left(
\frac{\beta\left(  p+\beta-1\right)  }{p}\right)  ^{2}\frac{2p}{\beta\left(
2\beta-1\right)  }+\beta^{2}\frac{2\left(  1-p\right)  }{\beta\left(
2\beta-1\right)  }\\
&  \ \ \ \ \ \ \ \ \ \ \ \ \ \ \ \ \ \ \ \ -\frac{\beta\left(  p+\beta
-1\right)  }{p}\frac{\beta}{p}\frac{2p}{\beta}+\left(  \frac{\beta}{p}\right)
^{2}p%
\begin{array}
[c]{c}%
=
\end{array}
\frac{1}{p}\frac{\beta^{2}}{2\beta-1},
\end{align*}
as desired.
\end{proof}

\begin{proposition}
\textbf{\label{prop5}}Assume that cdf's $F$ and $G$ satisfy the first-order
condition of regular variation $\left(  \ref{RV}\right)  .$ Then, we have%
\[
\overline{H}^{\left(  1\right)  }\left(  Z_{n:n}\right)  /\overline{H}%
_{n}^{\left(  1\right)  }\left(  Z_{n:n}\right)  =O_{\mathbf{P}}\left(
1\right)  ,\text{as }n\rightarrow\infty.
\]

\end{proposition}

\begin{proof}
Let $\mathbb{U}_{n}\left(  t\right)  :=n^{-1}\sum_{i=1}^{n}\mathbb{I}%
_{\left\{  U_{i}\leq t\right\}  }$ be the empirical cdf pertaining to the
sample $\left(  U_{i}\right)  _{i\geq1}$ of iid uniform-$\left(  0,1\right)  $
rv's defined in $\left(  \ref{unif}\right)  .$ From statement (2.14) in
\cite{Deheuvels1996}, we have almost surely (a.s.)%
\[
H_{n}^{\left(  1\right)  }\left(  t\right)  =\mathbb{U}_{n}\left(  H^{\left(
1\right)  }\left(  t\right)  \right)  ,\text{ for }0<H^{\left(  1\right)
}\left(  t\right)  <\theta,
\]
where $\theta:=H^{\left(  1\right)  }\left(  \infty\right)  \neq0.$ This
implies that
\[
\overline{H}_{n}^{\left(  1\right)  }\left(  t\right)  =\mathbb{U}_{n}\left(
H^{\left(  1\right)  }\left(  \infty\right)  \right)  -\mathbb{U}_{n}\left(
H^{\left(  1\right)  }\left(  t\right)  \right)
\]
a.s. too. On the other hand, since $H^{\left(  1\right)  }\left(
\infty\right)  <H^{\left(  1\right)  }\left(  t\right)  $ then
\[
\mathbb{U}_{n}\left(  H^{\left(  1\right)  }\left(  \infty\right)  \right)
-\mathbb{U}_{n}\left(  H^{\left(  1\right)  }\left(  t\right)  \right)
=\mathbb{U}_{n}\left(  \overline{H}^{\left(  1\right)  }\left(  t\right)
\right)  .
\]
Thus $\overline{H}_{n}^{\left(  1\right)  }\left(  t\right)  =\mathbb{U}%
_{n}\left(  \overline{H}^{\left(  1\right)  }\left(  t\right)  \right)  $ a.s.
as well. Note that we have%
\[
n\overline{H}\left(  Z_{n:n}\right)  =1+o_{\mathbf{P}}\left(  1\right)
=nU_{1:n},
\]
then, thanks to Proposition \ref{prop6}$,$ we deduce that $\overline
{H}^{\left(  1\right)  }\left(  Z_{n:n}\right)  =pn^{-1}\left(
1+o_{\mathbf{P}}\left(  1\right)  \right)  .$ Hence, we get
\[
\overline{H}^{\left(  1\right)  }\left(  Z_{n:n}\right)  =\left(
1+o_{\mathbf{P}}\left(  1\right)  \right)  pU_{1:n}.
\]
This means that for sufficiently small $\epsilon>0,$ we have $\left(
1-\epsilon\right)  pU_{1:n}<\overline{H}^{\left(  1\right)  }\left(
Z_{n:n}\right)  <1,$ with probability tending to one, which means that
$\overline{H}^{\left(  1\right)  }\left(  Z_{n:n}\right)  \in\left(
U_{1:n},1\right)  ,$ for all large $n.$ From Theorem 1 in \cite{SW86} (page
412): for large $n,$ we have%
\[
\sup_{U_{1:n}\leq s\leq1}s/\mathbb{U}_{n}\left(  s\right)  =O_{\mathbf{P}%
}\left(  1\right)  .
\]
Therefore, by letting $s=\overline{H}^{\left(  1\right)  }\left(
Z_{n:n}\right)  ,$ we get%
\[
\frac{\overline{H}^{\left(  1\right)  }\left(  Z_{n:n}\right)  }{\overline
{H}_{n}^{\left(  1\right)  }\left(  Z_{n:n}\right)  }=\frac{\overline
{H}^{\left(  1\right)  }\left(  Z_{n:n}\right)  }{\mathbb{U}_{n}\left(
\overline{H}^{\left(  1\right)  }\left(  Z_{n:n}\right)  \right)
}=O_{\mathbf{P}}\left(  1\right)  ,
\]
which completes the proof.
\end{proof}

\begin{proposition}
\textbf{\label{prop6}}Under the assumptions of Proposition \ref{prop5}, we
have
\[
\frac{\overline{H}^{\left(  1\right)  }\left(  Z_{n:n}\right)  }{\overline
{H}\left(  Z_{n:n}\right)  }=p\left(  1+o_{\mathbf{P}}\left(  1\right)
\right)  ,\text{ as }n\rightarrow\infty.
\]

\end{proposition}

\begin{proof}
This is a consequence of Lemma 4.1 in \cite{BMN-2015}.
\end{proof}

\begin{proposition}
\textbf{\label{prop7}}Under the assumptions of Proposition \ref{prop5}, we
have%
\[
\sup_{x\leq Z_{n:n}}\left\vert \frac{\overline{F}_{n}^{\left(  KM\right)
}\left(  x\right)  }{\overline{F}_{n}^{\left(  NA\right)  }\left(  x\right)
}-1\right\vert =O_{\mathbf{P}}\left(  n^{-1}\right)  ,\text{ as }%
n\rightarrow\infty.
\]

\end{proposition}

\begin{proof}
For $x\leq Z_{n:n},$ we have%
\begin{align*}
\overline{F}_{n}^{\left(  KM\right)  }\left(  x\right)   &  =\prod
_{Z_{i:n}\leq x}\left(  1-\frac{1}{n-i+1}\right)  ^{\delta_{\left[
i:n\right]  }}\\
&  =\exp\left\{  \sum_{i=1}^{n}\mathbb{I}_{\left\{  Z_{i:n}\leq x\right\}
}\delta_{\left[  i:n\right]  }\log\left(  1-\frac{1}{n-i+1}\right)  \right\}
\\
&  =\exp\left\{  n\int_{0}^{x}\log\left(  1-\frac{1}{n\overline{H}_{n}\left(
t\right)  }\right)  dH_{n}^{\left(  1\right)  }\left(  t\right)  \right\}  .
\end{align*}
Using Taylor's expansion, as $n\rightarrow\infty,$ yields that%
\[
\log\left(  1-\frac{1}{n\overline{H}_{n}\left(  t\right)  }\right)  =-\frac
{1}{n\overline{H}_{n}\left(  t\right)  }+O_{\mathbf{P}}\left(  \frac{1}%
{n^{2}\left[  \overline{H}_{n}\left(  t\right)  \right]  ^{2}}\right)  ,\text{
for }0<t<x,
\]
It follows that%
\[
\overline{F}_{n}^{\left(  KM\right)  }\left(  x\right)  =\exp\left(  -\int%
_{0}^{x}\frac{dH_{n}^{\left(  1\right)  }\left(  t\right)  }{\overline{H}%
_{n}\left(  t\right)  }\right)  \exp\left(  O_{\mathbf{P}}\left(
n^{-2}\right)  \int_{0}^{x}\frac{dH_{n}^{\left(  1\right)  }\left(  t\right)
}{\left[  \overline{H}_{n}\left(  t\right)  \right]  ^{2}}\right)  ,
\]
which equals $I_{n1}\left(  x\right)  \left(  1+O_{\mathbf{P}}\left(
n^{-2}\right)  I_{n2}\left(  x\right)  \right)  ,$ as $n\rightarrow\infty,$
where
\[
I_{n1}\left(  x\right)  :=\exp\left(  -\int_{0}^{x}\frac{dH_{n}^{\left(
1\right)  }\left(  t\right)  }{\overline{H}_{n}\left(  t\right)  }\right)
\text{ and }I_{n2}\left(  x\right)  :=\int_{0}^{x}\frac{dH_{n}^{\left(
1\right)  }\left(  t\right)  }{\left[  \overline{H}_{n}\left(  t\right)
\right]  ^{2}}.
\]
Observe that%
\begin{align*}
I_{n1}\left(  x\right)   &  =\exp\left(  -\int_{0}^{x}\frac{dH_{n}^{\left(
1\right)  }\left(  t\right)  }{\overline{H}_{n}\left(  t\right)  }\right)
=\exp\left(  -\sum_{i=1}^{n}\mathbb{I}_{\left\{  Z_{j:n}\leq x\right\}  }%
\frac{\delta_{\left[  j:n\right]  }}{\overline{H}_{n}\left(  Z_{j:n}\right)
}\right)  \\
&  =\prod_{Z_{j:n}\leq x}\exp\left\{  -\frac{\delta_{\left[  j:n\right]  }%
}{n-j+1}\right\}  =\overline{F}_{n}^{\left(  NA\right)  }\left(  x\right)  .
\end{align*}
It remains to show that $I_{n2}\left(  x\right)  =O_{\mathbf{P}}\left(
n\right)  ,$ uniformly over $0<x\leq Z_{n:n},$ for all large $n.$ Recall that
$\overline{H}_{n}\left(  t\right)  =\overline{H}_{n}^{\left(  1\right)
}\left(  t\right)  +\overline{H}_{n}^{\left(  0\right)  }\left(  t\right)  ,$
then $\overline{H}_{n}\left(  t\right)  >\overline{H}_{n}^{\left(  1\right)
}\left(  t\right)  ,$ it follows that%
\[
0<\int_{0}^{x}\frac{dH_{n}^{\left(  1\right)  }\left(  t\right)  }{\left[
\overline{H}_{n}\left(  t\right)  \right]  ^{2}}<\int_{0}^{x}\frac
{dH_{n}^{\left(  1\right)  }\left(  t\right)  }{\left[  \overline{H}%
_{n}^{\left(  1\right)  }\left(  t\right)  \right]  ^{2}}=\frac{1}%
{\overline{H}_{n}^{\left(  1\right)  }\left(  x\right)  }-\frac{1}%
{\overline{H}_{n}^{\left(  1\right)  }\left(  0\right)  }<\frac{1}%
{\overline{H}_{n}^{\left(  1\right)  }\left(  x\right)  }.
\]
Note that $\overline{H}_{n}^{\left(  1\right)  }$ is deceasing, it follows
that $1/\overline{H}_{n}^{\left(  1\right)  }\left(  x\right)  \leq
1/\overline{H}_{n}^{\left(  1\right)  }\left(  Z_{n:n}\right)  .$ On the other
hand, from Propositions \ref{prop5} and \ref{prop6}, we have
\[
\overline{H}^{\left(  1\right)  }\left(  Z_{n:n}\right)  /\overline{H}%
_{n}^{\left(  1\right)  }\left(  Z_{n:n}\right)  =O_{\mathbf{P}}\left(
1\right)  \text{ and }\overline{H}\left(  Z_{n:n}\right)  /\overline
{H}^{\left(  1\right)  }\left(  Z_{n:n}\right)  =p\left(  1+o_{\mathbf{P}%
}\left(  1\right)  \right)  ,
\]
it follows that $1/\overline{H}_{n}^{\left(  1\right)  }\left(  x\right)
=O_{\mathbf{P}}\left(  1/\overline{H}\left(  Z_{n:n}\right)  \right)  .$
Recall that\ $n\overline{H}\left(  Z_{n:n}\right)  \overset{\mathbf{P}%
}{\rightarrow}1,$ therefore, $1/\overline{H}_{n}^{\left(  1\right)  }\left(
x\right)  =O_{\mathbf{P}}\left(  n\right)  $ and so $I_{n2}\left(  x\right)
=O_{\mathbf{P}}\left(  n\right)  $ as well. Consequently, we have%
\[
\overline{F}_{n}^{\left(  KM\right)  }\left(  x\right)  =\overline{F}%
_{n}^{\left(  NA\right)  }\left(  x\right)  \left(  1+O_{\mathbf{P}}\left(
n^{-1}\right)  \right)  ,
\]
uniformly over $x\geq Z_{n:n}$ as intended.
\end{proof}

\newpage

\section{\textbf{Appendix B \label{Appendix B}}}%

\begin{figure}[h]%
\centering
\includegraphics[
height=2.7025in,
width=5.6593in
]%
{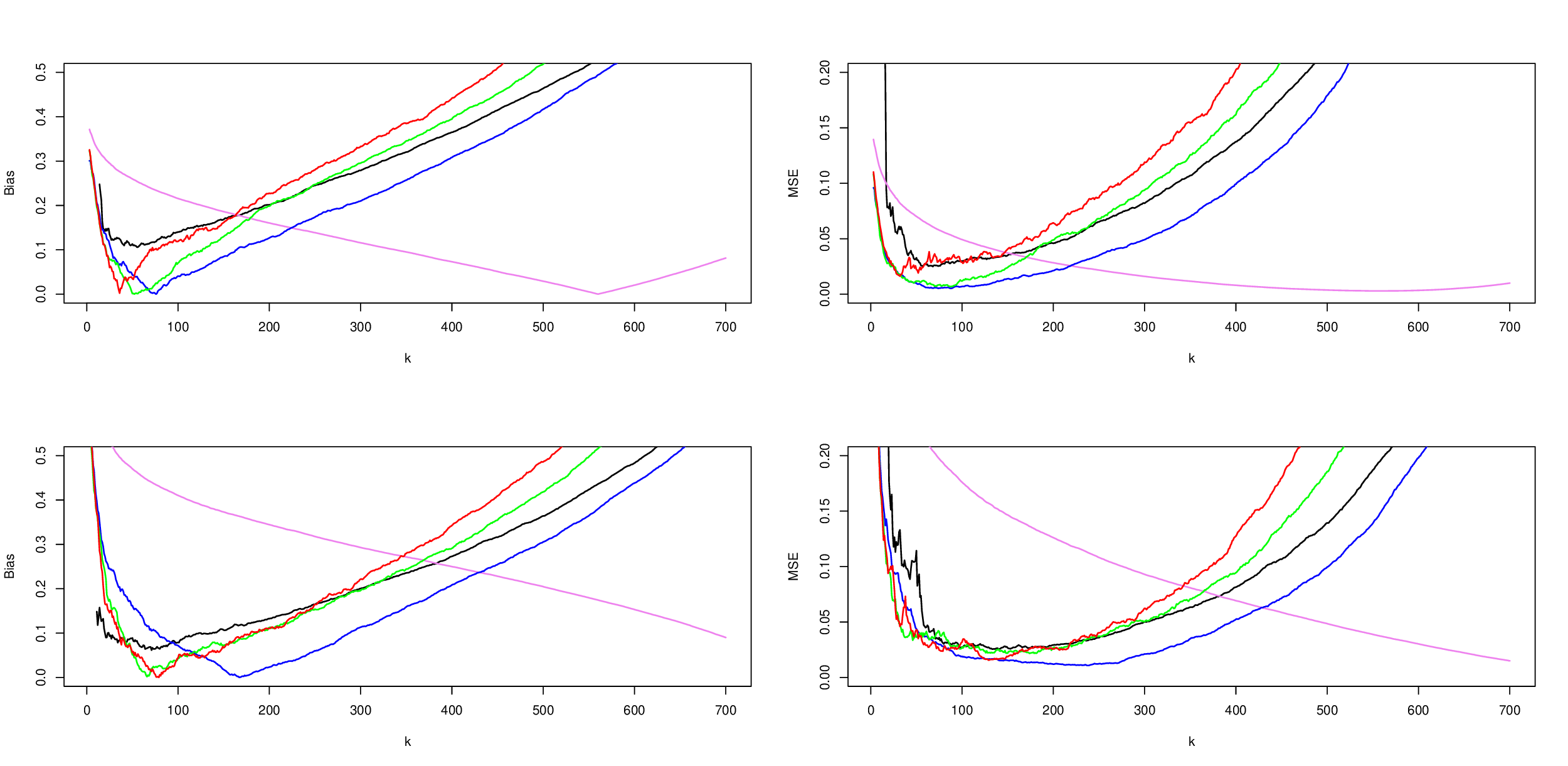}%
\caption{Bias (left panel) and MSE (right panel) of $\protect\widehat{\gamma
}_{1,k}^{\left(  NA\right)  }\left(  1.01\right)  $ (blue line),
$\protect\widehat{\gamma}_{1,k}^{\left(  NA\right)  }\left(  1.5\right)  $
(green line), $\protect\widehat{\gamma}_{1,k}^{\left(  NA\right)  }\left(
2\right)  $ (red line), $\protect\widehat{\gamma}_{1,k}^{\left(  MNS\right)
}$ (purple line) and $\protect\widehat{\gamma}_{1,k}^{\left(  EFG\right)  }$
(black line) based on $2000$ samples of size $1000$ from Burr model censored
by Burr for $\gamma_{1}=0.4$ (top) and $\gamma_{1}=0.7$ (bottom), with
$p=0.30.$ }%
\label{fig1}%
\end{figure}
%

\begin{figure}[h]%
\centering
\includegraphics[
height=2.7129in,
width=5.5971in
]%
{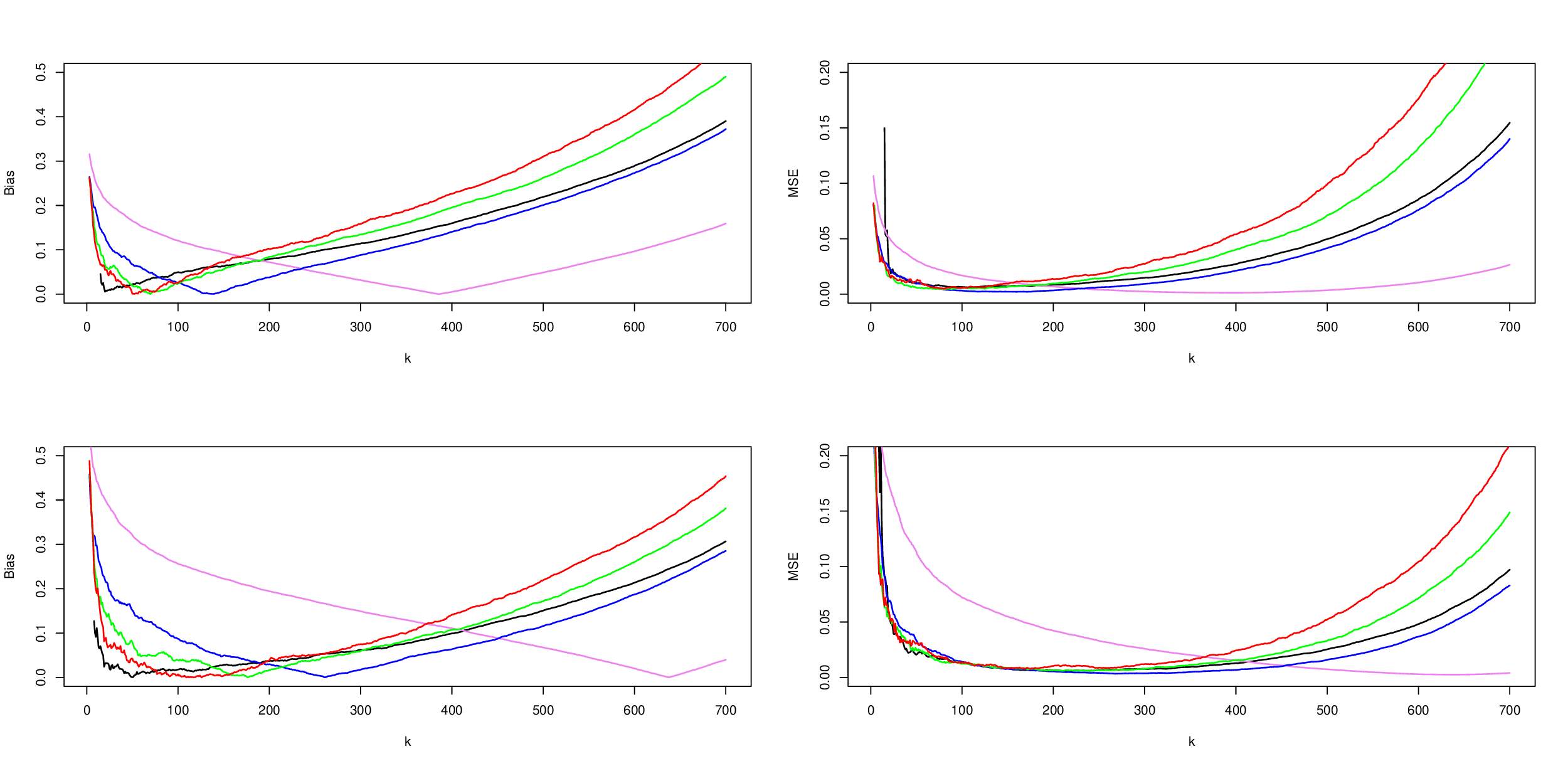}%
\caption{Bias (left panel) and MSE (right panel) of $\protect\widehat{\gamma
}_{1,k}^{\left(  NA\right)  }\left(  1.01\right)  $ (blue line),
$\protect\widehat{\gamma}_{1,k}^{\left(  NA\right)  }\left(  1.5\right)  $
(green line), $\protect\widehat{\gamma}_{1,k}^{\left(  NA\right)  }\left(
2\right)  $ (red line), $\protect\widehat{\gamma}_{1,k}^{\left(  MNS\right)
}$ (purple line) and $\protect\widehat{\gamma}_{1,k}^{\left(  EFG\right)  }$
(black line) based on $2000$ samples of size $1000$ from Burr model censored
by Burr for $\gamma_{1}=0.4$ (top) and $\gamma_{1}=0.7$ (bottom) with
$p=0.50.$}%
\label{fig2}%
\end{figure}
%

\begin{figure}[h]%
\centering
\includegraphics[
height=2.6714in,
width=5.5971in
]%
{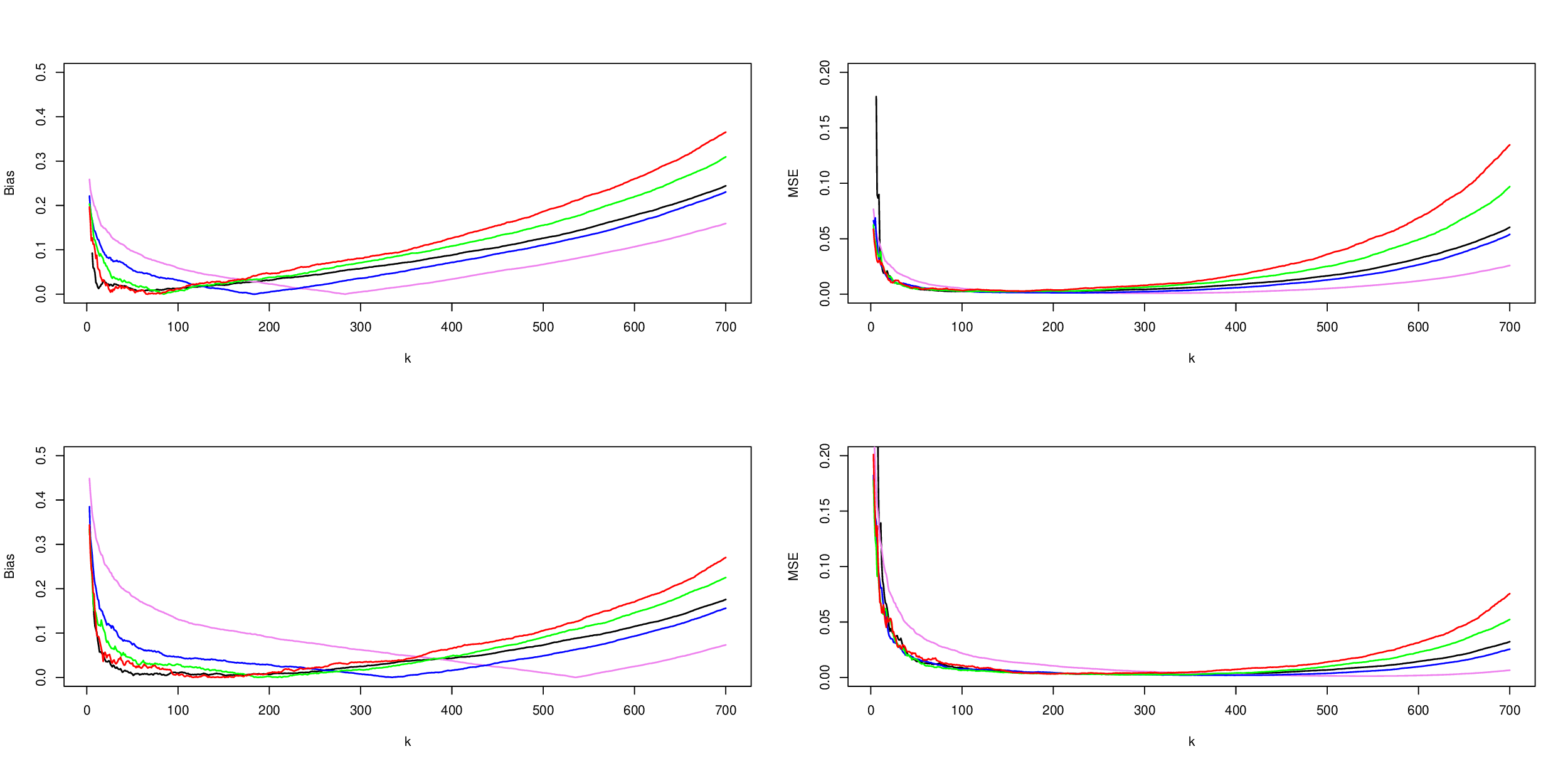}%
\caption{Bias (left panel) and MSE (right panel) of $\protect\widehat{\gamma
}_{1,k}^{\left(  NA\right)  }\left(  1.01\right)  $ (blue line),
$\protect\widehat{\gamma}_{1,k}^{\left(  NA\right)  }\left(  1.5\right)  $
(green line), $\protect\widehat{\gamma}_{1,k}^{\left(  NA\right)  }\left(
2\right)  $ (red line), $\protect\widehat{\gamma}_{1,k}^{\left(  MNS\right)
}$ (purple line) and $\protect\widehat{\gamma}_{1,k}^{\left(  EFG\right)  }$
(black line) based on $2000$ samples of size $1000$ from Burr model censored
by Burr for $\gamma_{1}=0.4$ (top) and $\gamma_{1}=0.7$ (bottom), with
$p=0.70.$}%
\label{fig3}%
\end{figure}
%

\begin{figure}[h]%
\centering
\includegraphics[
height=2.7025in,
width=5.5763in
]%
{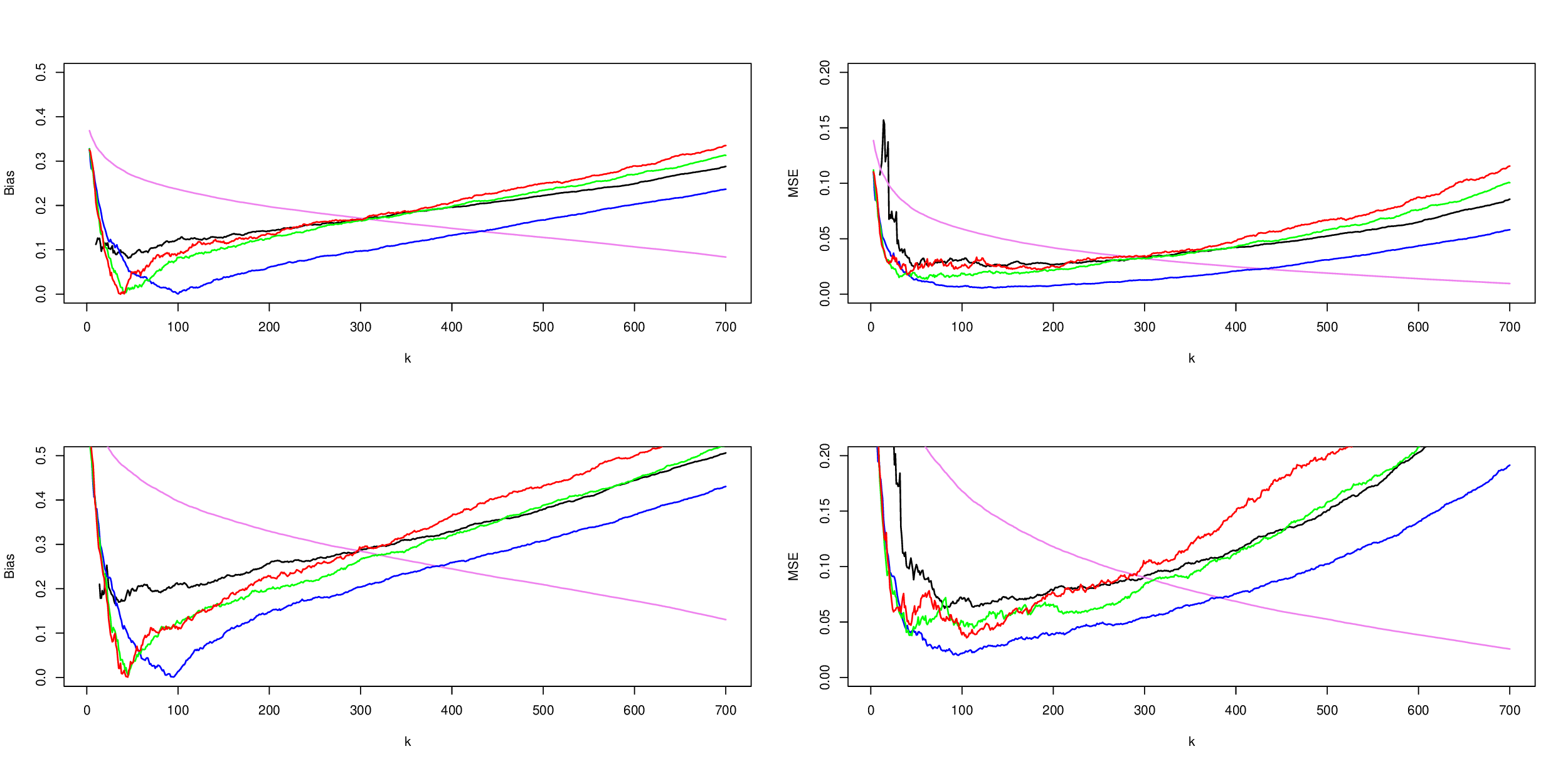}%
\caption{Bias (left panel) and MSE (right panel) of $\protect\widehat{\gamma
}_{1,k}^{\left(  NA\right)  }\left(  1.01\right)  $ (blue line),
$\protect\widehat{\gamma}_{1,k}^{\left(  NA\right)  }\left(  1.5\right)  $
(green line), $\protect\widehat{\gamma}_{1,k}^{\left(  NA\right)  }\left(
2\right)  $ (red line), $\protect\widehat{\gamma}_{1,k}^{\left(  MNS\right)
}$ (purple line) and $\protect\widehat{\gamma}_{1,k}^{\left(  EFG\right)  }$
(black line) based on $2000$ samples of size $1000$ from Fr\'{e}chet model
censored by Fr\'{e}chet for $\gamma_{1}=0.4$ (top) and $\gamma_{1}=0.7$
(bottom), with $p=0.30.$}%
\label{fig4}%
\end{figure}
%

\begin{figure}[h]%
\centering
\includegraphics[
height=2.7233in,
width=5.5763in
]%
{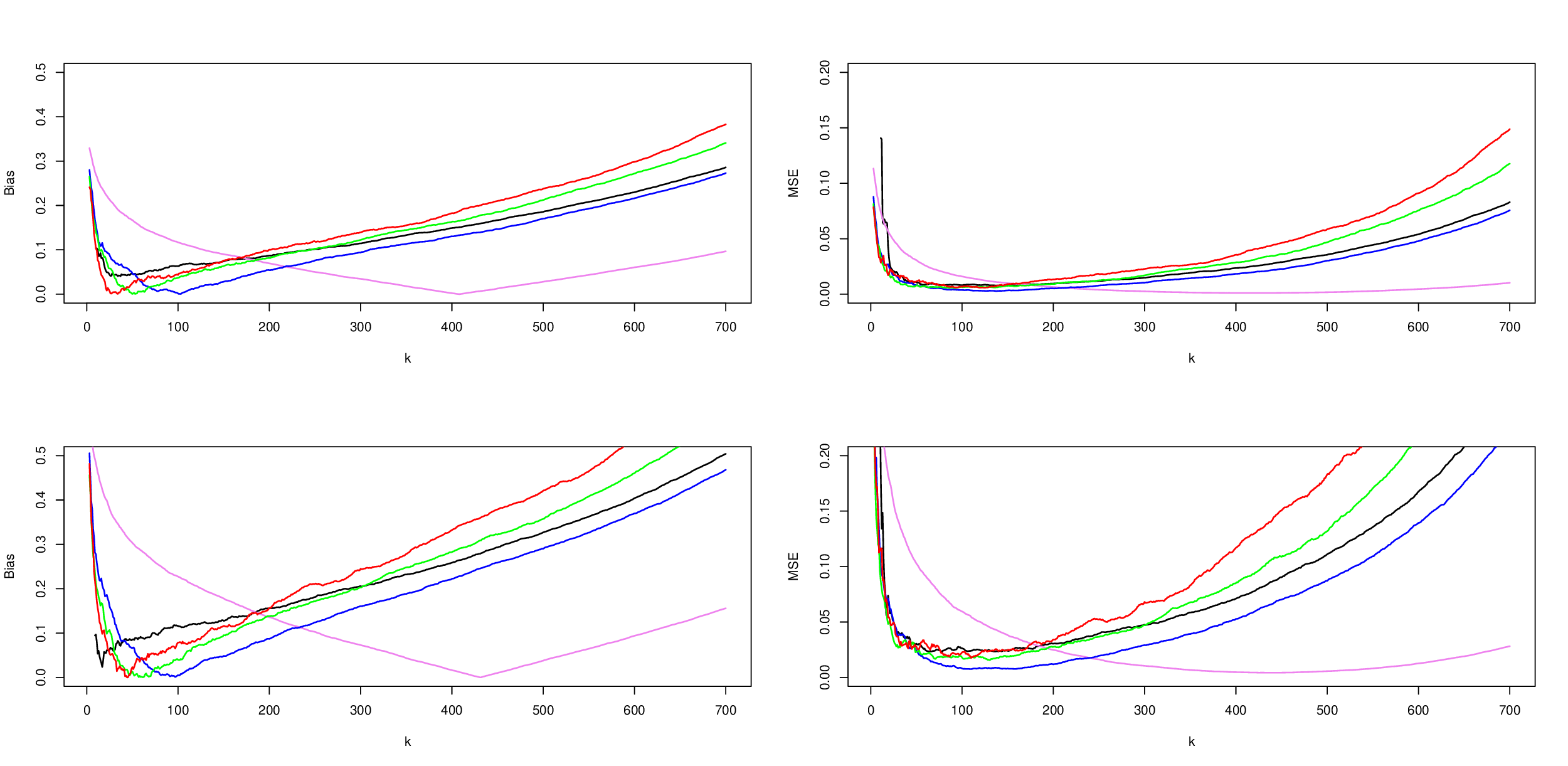}%
\caption{Bias (left panel) and MSE (right panel) of $\protect\widehat{\gamma
}_{1,k}^{\left(  NA\right)  }\left(  1.01\right)  $ (blue line),
$\protect\widehat{\gamma}_{1,k}^{\left(  NA\right)  }\left(  1.5\right)  $
(green line), $\protect\widehat{\gamma}_{1,k}^{\left(  NA\right)  }\left(
2\right)  $ (red line), $\protect\widehat{\gamma}_{1,k}^{\left(  MNS\right)
}$ (purple line) and $\protect\widehat{\gamma}_{1,k}^{\left(  EFG\right)  }$
(black line) based on $2000$ samples of size $1000$ from Fr\'{e}chet model
censored by Fr\'{e}chet for $\gamma_{1}=0.4$ (top) and $\gamma_{1}=0.7$
(bottom), with $p=0.50.$}%
\label{fig5}%
\end{figure}
%

\begin{figure}[h]%
\centering
\includegraphics[
height=2.7129in,
width=5.4829in
]%
{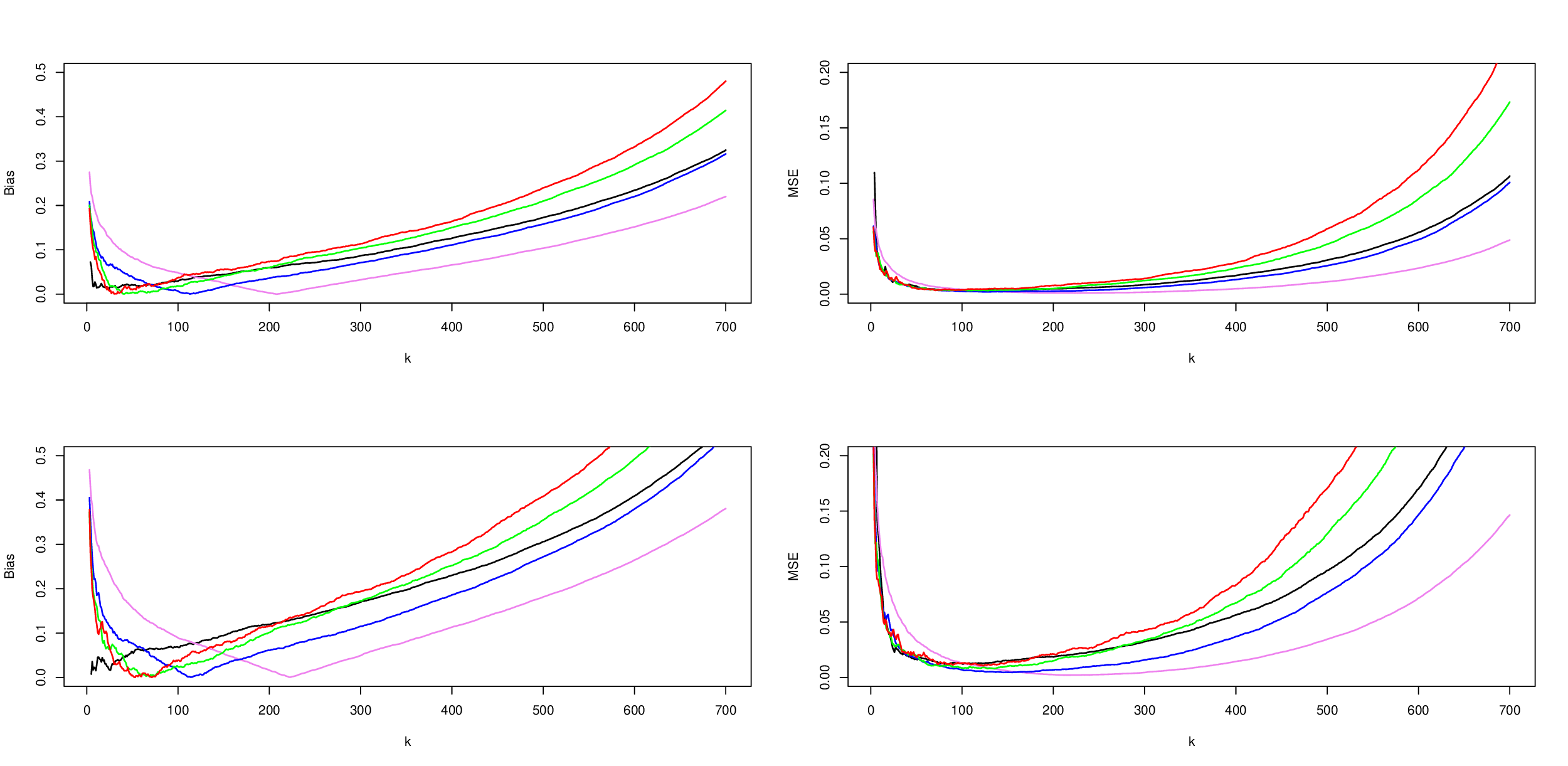}%
\caption{Bias (left panel) and MSE (right panel) of $\protect\widehat{\gamma
}_{1,k}^{\left(  NA\right)  }\left(  1.01\right)  $ (blue line),
$\protect\widehat{\gamma}_{1,k}^{\left(  NA\right)  }\left(  1.5\right)  $
(green line), $\protect\widehat{\gamma}_{1,k}^{\left(  NA\right)  }\left(
2\right)  $ (red line), $\protect\widehat{\gamma}_{1,k}^{\left(  MNS\right)
}$ (purple line) and $\protect\widehat{\gamma}_{1,k}^{\left(  EFG\right)  }$
(black line) based on $2000$ samples of size $1000$ from Fr\'{e}chet model
censored by Fr\'{e}chet for $\gamma_{1}=0.4$ (top) and $\gamma_{1}=0.7$
(bottom), with $p=0.70.$}%
\label{fig6}%
\end{figure}

\end{document}